\UseRawInputEncoding
\documentclass[preprint]{elsarticle}

\usepackage{graphicx}
\usepackage{amsmath}
\usepackage{url}
\usepackage{mathptmx}
\usepackage{multirow}

\usepackage{times}
\usepackage{pstricks,pstricks-add}
\usepackage{latexsym,amsmath,amssymb,amsfonts}

\newrgbcolor{Lgray}{0.75 0.75 0.75}
\newtheorem{definition}{Definition}
\newtheorem{theorem}{Theorem}
\newtheorem{corollary}{Corollary}

\newproof{proof}{Proof}

\def\acurve#1 #2*#3*{%
\pscurve#2#3
\psecurve[linestyle=none]{#1}#2}

\psset{arrowsize=3pt 3,arrowinset=0.65,arrowlength=0.8}

\begin{document}

\title{Variable stepsize SDIMSIMs
for ordinary differential equations}
\author{A. Jalilian}
\ead{a.jalilian@tabrizu.ac.ir}

\author{A.~Abdi\corref{cor}}
\ead{a\_abdi@tabrizu.ac.ir}

\author{G.~Hojjati}
\ead{ghojjati@tabrizu.ac.ir}

\cortext[cor]{Corresponding author}

\address{Faculty of Mathematical Sciences, University of Tabriz,
Tabriz, Iran}
\begin{abstract}
Second derivative general linear methods (SGLMs) have been already
implemented in a variable stepsize environment using Nordsieck
technique. In this paper, we introduce variable stepsize SGLMs
directly on nonuniform grid. By deriving the order conditions of
the proposed methods of order $p$ and stage order $q=p$, some
explicit examples of these methods up to order four are given. By
some numerical experiments, we show the efficiency of the proposed
methods in solving nonstiff problems and confirm the theoretical
order of convergence.
\end{abstract}

\maketitle
\setcounter{equation}{0} \setcounter{definition}{0}
\setcounter{theorem}{0}
\section{Introduction}\label{Sec1}
Efficiency of the constructed methods for the numerical solution
of initial value problem of ordinary differential equations (ODEs)
\begin{align}\label{ODEs}
\left\{\begin{array}{l}
y'(x) = f(y(x)),\hspace{1cm} x\in[x_0,X],\\[2mm]
y(x_0)=y_0,
\end{array}
\right.
\end{align}
where $f:\mathbb{R}^m\rightarrow \mathbb{R}^m$ is sufficiently
smooth function and $m$ is the dimensional of system, does depend
on designing of their implementation issues. Indeed, reliable
codes are developed based on the methods with desirable features
and more sophisticated implementation strategies. Among all the
classes of numerical methods, the class of general linear methods
(GLMs)---due to the fact of their comprehensive structure---is one
of the best. In fact, GLMs are multivalue-multistage methods which
include linear multistep and Rungee--Kutta methods as special
cases. These methods which were introduced by Butcher
\cite{But1966,But2016} take the form
\begin{align}\label{GLMs}
\left\{\begin{array}{l} \displaystyle
Y_{i}^{[n+1]}=h\sum_{j=1}^{s}a_{ij}f(Y_{j}^{[n+1]})
+\sum_{j=1}^{r}u_{ij}y_{j}^{[n]},\quad i=1,\ldots,s,\\
\displaystyle y_{i}^{[n+1]}= h\sum_{j=1}^{s}b_{ij}f(Y_{j}^{[n+1]})
+\sum_{j=1}^{r}v_{ij}y_{j}^{[n]},\quad i=1,\ldots,r,
\end{array}
\right.
\end{align}
where $n=0,1,\ldots,N-1$, $Nh=X-x_{0}$, $s$ is the number of
internal stages and $r$ is the number of input and output values.
Here, $Y_i^{[n+1]}$, $i=1,2,\ldots,s$, are an approximation of the
stage order $q$ to the solution $y$ at the points $x_{n}+c_ih$,
with $c=[c_{1}\hspace{2mm} c_2\hspace{2mm}\cdots\hspace{2mm}
c_{s}]^{T}$ as the abscissa vector. Moreover, the values
$y_i^{[n+1]}$, $i=1,2,\ldots,r$, are approximations of order $p$
to the linear combinations of scaled derivatives of the solution
$y$ at the points $x_{n+1}$, i.e.,
\[
y_i^{[n+1]}=\alpha_{i0}y(x_{n+1})+\alpha_{i1}hy'(x_{n+1})
+\alpha_{i2}h^2y''(x_{n+1})+\cdots+\alpha_{ip}h^py^{(p)}(x_{n+1})
+O(h^{p+1}),
\]
for some real numbers
$\alpha_{i1},\alpha_{i2},\ldots,\alpha_{ip}$. These values denote
the computed quantities at step $n+1$, which also represent the
incoming values for step $n+2$. Diagonally implicit multistage
integration methods (DIMSIMs) as a large subclass of GLMs have
been studied and discussed by Butcher and Jackiewicz, for instance
in \cite{ButJac1997,ButJack1998,ButJackMitt1997,Jackbook2009}.
Implementation issues for DIMSIMs have been studied in
\cite{ButJack1997,ButJack2001,ButPod2006} and efficient
\textsc{Matlab} codes \texttt{dim18.m} and \texttt{dim13s.m} based
on these methods using the Nordsieck technique have been developed
in \cite{ButJac1999,Jackbook2002}. Furthermore, another reliable
code has been developed based on the Nordsieck GLMs in which the
matrix $W:=[\alpha_{ij}]$ is equal to the identity matrix, i.e.,
the input and output vectors $y^{[n-1]}$ and $y^{[n]}$ approximate
the Nordsieck vector $z(x,h)$ defined by
\[
z(x,h)=\left[%
\begin{array}{c}
  y(x) \\[2mm]
  hy'(x) \\
  \vdots \\[1mm]
  h^py^{(p)}(x) \\
\end{array}%
\right],
\]
at the points $x=x_{n-1}$ and $x=x_n$, respectively. Also, in
\cite{Jack1995}, the authors studied a class of variable stepsize
DIMSIMs which provides an alternative to the Nordsieck technique
of changing the stepsize of integration. This approach is based on
the derivation of the methods based directly on nonuniform grid so
that the coefficients matrices $A:=[a_{ij}]$, $U:=[u_{ij}]$,
$B:=[b_{ij}]$, $V:=[v_{ij}]$, the vector $c$ as well as the matrix
$W$ depend, in general, on the ratios of the current stepsize and
the past stepsizes.

To increase the chances of finding methods with high order
accuracy together with desirable stability properties, higher
derivatives, especially second derivative, of the solution can be
used. Therefore, several researches have been focused on the
construction of the methods incorporating the second derivative of
the solution, see, for instance,
\cite{Cash1981,Chan2010,Enrigh1972,Hairer,Hoj2006}. Trying to
construct second derivative methods in a unified structure and
with efficient properties leads to the second derivative general
linear methods (SGLMs) where were first introduced by Butcher and
Hojjati in \cite{ButHoj2005} and more investigated by Abdi et al.
in
\cite{AbdB2018,AbdBra2018,AbdCon,AbdHoj2011a,AbdHoj2011,AbdHoj2015}.
Denoting the function $g:=f'(\cdot)f(\cdot)$ as the second
derivative of the solution $y$ and using of the previous
notations, the SGLMs take the form
\begin{align}\label{SGLMs}
\left\{\begin{array}{l} \displaystyle
Y_{i}^{[n+1]}=h\sum_{j=1}^{s}a_{ij}f(Y_{j}^{[n+1]})+h^{2}\sum_{j=1}^{s}\overline{a}_{ij}g(Y_{j}^{[n+1]})+ \sum_{j=1}^{r}u_{ij}y_{j}^{[n]},\quad i=1,\ldots,s,\\
\displaystyle y_{i}^{[n+1]}=
h\sum_{j=1}^{s}b_{ij}f(Y_{j}^{[n+1]})+h^{2}\sum_{j=1}^{s}\overline{b}_{ij}g(Y_{j}^{[n+1]})+
\sum_{j=1}^{r}v_{ij}y_{j}^{[n]},\quad i=1,\ldots,r.
\end{array}
\right.
\end{align}
The SGLM \eqref{SGLMs} has order $p$ and stage order $q=p$ if and
only if \cite{AbdHoj2011}
\begin{equation}\label{OC.SGLMs}
\begin{array}{r}
C=ACK+\overline{A}CK^2+UW,\\[2mm]
WE=BCK+\overline{B}CK^2+VW,
\end{array}
\end{equation}
where the matrices $C$, $K$, and $E$ are determined by
\begin{align*}
C&:=\left[e\hspace{2mm}{c}\hspace{2mm}\frac{c^2}{2!}\hspace{2mm}
\cdots\hspace{2mm}
\frac{c^p}{p!}\right]\in\mathbb{R}^{s\times(p+1)},\quad
K:=[0\hspace{2mm}e_1\hspace{2mm}e_2\hspace{2mm}\cdots\hspace{2mm}e_p]
\in\mathbb{R}^{(p+1)\times(p+1)},\\[2mm]
E&:=\exp(K)\in\mathbb{R}^{(p+1)\times(p+1)}
\end{align*}
with
$e=[1\hspace{3mm}1\hspace{3mm}\cdots\hspace{3mm}1]^T\in\mathbb{R}^s$
and $e_j$ as the $j$th vector of the canonical basis in
$\mathbb{R}^{p+1}$.

Construction and implementation of the Nordsieck SGLMs have been
discussed in \cite{AbdB2018,AbdHoj2015}. Furthermore, in
\cite{AbdCon}, the authors have investigated the implementation of
SGLMs in a variable stepsize environment using the Nordsieck
technique and the practical \textsc{Matlab} code \texttt{SGLM4.m}
based on an $L$--stable SGLM of order four has been developed.

Second derivative diagonally implicit multistage integration
methods (SDIMSIMs) as a subclass of SGLMs have been introduced in
\cite{AbdHoj2011a} in which usually $p=q=r=s$ and the matrix $V$
is a rank-one matrix. The spirit of this paper is to introduce
SDIMSIMs directly on nonuniform grid. For such methods, it is not
required to update the input vector for the new stepsize;
actually, the output vector of the last step can be directly used
in the next step as the input vector.

The rest of the paper is organized along the following lines. In
Section \ref{Sec2}, we formulate the SDIMSIMs \eqref{SGLMs} as
variable stepsize methods directly on nonuniform grid. Then, the
order conditions of these methods in the case that the stage order
is equal to the order of the methods, are derived. Section
\ref{Sec3} is devoted to the construction of such methods in a
special class up to order four. To show the efficiency of the
constructed methods, some numerical results are provided in
Section \ref{Sec4}.

\setcounter{equation}{0} \setcounter{definition}{0}
\setcounter{theorem}{0}
\section{Variable stepsize formulation of SDIMSIMs}\label{Sec2}
The purpose of this section is to formulate the SDIMSIMs
\eqref{SGLMs} as variable stepsize (VS) methods in which the
coefficients matrices of the methods, the abscissa vector $c$ and
the parameteres $\alpha_{ik}$ depend on the ratios of the current
stepsize and the previous stepsizes. To accomplish this, for a
given integer $\rho$, consider a nonuniform grid
\begin{align}\label{mesh}
x_{-\rho}<\cdots<x_{-1}<x_{0}<x_{1}<\cdots<x_{N},\quad x_{N}>X,
\end{align}
with $h_{n}=x_{n+1}-x_{n}$, $n= -\rho,-\rho+1,\ldots,N+1$. Let us
denote $\sigma_{n,i}=h_{n-i}/h_{n}$, $i=1,2,\ldots,\rho$ and
$\sigma_{n}=[\sigma_{n,1}\hspace{2mm}\sigma_{n,2}\hspace{2mm}\ldots\hspace{2mm}\sigma_{n,\rho}]^{T}$,
$n= -\rho,-\rho+1,\ldots,N+1$. Here the artificial points
$x_{-\rho},\ldots,x_{-1}$ are introduced only for notational
convenience and we will start the integration process from
$x_\rho$ which only requires the points
$\left\{x_n\right\}_{n=0}^{N}$. In fact, we discuss VS SDIMSIMs in
the form
\begin{align}\label{vs-SGLMs}
\left\{\begin{array}{l} \displaystyle
Y_{i}^{[n+1]}=h_n\sum_{j=1}^{s}a_{ij}(\sigma_n)f(Y_{j}^{[n+1]})
+h_n^{2}\sum_{j=1}^{s}\overline{a}_{ij}(\sigma_n)g(Y_{j}^{[n+1]})+
\sum_{j=1}^{r}u_{ij}(\sigma_n)y_{j}^{[n]},\quad i=1,\ldots,s,\\
\displaystyle y_{i}^{[n+1]}=
h_n\sum_{j=1}^{s}b_{ij}(\sigma_n)f(Y_{j}^{[n+1]})+
h_n^{2}\sum_{j=1}^{s}\overline{b}_{ij}(\sigma_n)g(Y_{j}^{[n+1]})+
\sum_{j=1}^{r}v_{ij}(\sigma_n)y_{j}^{[n]},\quad i=1,\ldots,r.
\end{array}
\right.
\end{align}
$n=0,1,\ldots,N-1$, where  $Y_i^{[n+1]}$, $i=1,2,\ldots,s$, are
approximations to $y(x_{n}+c_i (\sigma_{n})h_{n})$ and $I_m$
stands for the identity matrix of dimension $m$. Here, the
starting values $y_i^{[0]}$, $i=1,2,\ldots,r$, approximate the
linear combinations of $y(x_{-\rho}),y(x_{-\rho+1}), \ldots,
y(x_0)$ of order $p$. Defining the matrices
\begin{align*}
A(\sigma_{n})&:=[a_{ij}(\sigma_{n})] \in\mathbb{R}^{s\times
s},\hspace{3mm}\overline{A}(\sigma_{n}):=[\overline{a}_{ij}(\sigma_{n})]
\in\mathbb{R}^{s\times s}\hspace{3mm}
U(\sigma_{n}):=[u_{ij}(\sigma_{n})] \in\mathbb{R}^{s\times
r},\\[2mm]
B(\sigma_{n})&:=[b_{ij}(\sigma_{n})] \in\mathbb{R}^{r\times
s},\hspace{3Mm}
\overline{B}(\sigma_{n}):=[\overline{b}_{ij}(\sigma_{n})]
\in\mathbb{R}^{r\times s},\hspace{3mm}
V(\sigma_{n}):=[v_{ij}(\sigma_{n})] \in\mathbb{R}^{r\times r},
\end{align*}
the methods \eqref{vs-SGLMs} can be written in the compact form
\begin{equation}\label{VS.SGLMs}
\left\{\begin{array}{l}
Y^{[n+1]}=h_n(A(\sigma_n)\otimes I_m)f(Y^{[n+1]})+h_n^{2}(\overline{A}(\sigma_{n})\otimes I_m)g(Y^{[n+1]})+ (U(\sigma_{n})\otimes I_m) y^{[n]},\\[5mm]
y^{[n+1]}=h_n(B(\sigma_n)\otimes I_m)
f(Y^{[n+1]})+h_n^{2}(\overline{B}(\sigma_{n})\otimes
I_m)g(Y^{[n+1]})+ (V(\sigma_{n})\otimes I_m) y^{[n]}.
\end{array}
\right.
\end{equation}
Through the paper, we assume that $p=q=r=s$ and $V(\sigma_{n})$ is
a rank-one matrix, i.e. $V(\sigma_{n})=e\,v(\sigma_{n})^T$ where
$v(\sigma_{n})=[v_1(\sigma_{n})\hspace{2mm}v_2(\sigma_{n})\hspace{2mm}\ldots\hspace{2mm}v_r(\sigma_{n})]^T$,
and $v(\sigma_{n})^Te=1$. The latter guarantees that the method
\eqref{VS.SGLMs} is unconditionally zero--stable for any step size
pattern; indeed, applying the method \eqref{VS.SGLMs} to $y'=0$
leads to
\begin{align*}
y^{[n]}&=\Bigl(\prod_{j=0}^{n}V(\sigma_{n-j})\Bigr)^ny^{[0]}\\
&=\Bigl(\prod_{j=0}^{n}ev^T(\sigma_{n-j})\Bigr)^ny^{[0]}\\
&=\bigl(ev(\sigma_0)\bigr)^ny^{[0]},
\end{align*}
in which the matrix $ev(\sigma_0)$ is a rank-one matrix with
nonzero eigenvalue equal to one. Moreover, applying the method
\eqref{VS.SGLMs} to $y'=\xi y$ leads to
\begin{align*}
y^{[n]}&=M(z_n;\sigma_n)y^{[n-1]}\\
&=\prod_{j=1}^nM(z_j;\sigma_j)\cdot y^{[0]},
\end{align*}
in which the \emph{propagation matrix} $M(z_j;\sigma_j)$ is
defined by
\[
M(z_j;\sigma_j):=V(\sigma_j)+\bigl(z_jB(\sigma_j)+z_j^2
\overline{B}(\sigma_j)\bigr)\bigl(I-z_jA(\sigma_j)-z_j^2
\overline{A}(\sigma_j)\bigr)^{-1}U(\sigma_j),
\]
with $z_j=\xi h_j$. So, in order to analyze the absolute stability
of such methods, the boundedness of products of propagation
matrices must be examined. It can be based on the spectral radius
of products of propagation matrices which would be too complicated
for arbitrary values of $z_j$. Such analysis for VS BDF methods
has been studied for a special sequences of stepsize ratios in
\cite{cmr1993}. References are also made to
\cite{But2016,gz2001-1,gz2001-2,Jackbook2009} which include some
discussions of stability of VS methods in different classes.

To simplify the presentation, in what follows, we will assume that
$m=1$. To derive the order conditions of VS SDIMSIMs
\eqref{VS.SGLMs}, assume that the vector of stage values
$Y^{[n+1]}$ is an approximation of order $q$ to the vector
$z_1(x_{n+1}):=y(x_n + c(\sigma_n)h_n)$, $n=0,1,\ldots,N-1$, i.e.
\begin{equation}\label{VS.stage}
Y^{[n+1]}=z_1(x_{n+1})+O(h^{q+1}).
\end{equation}
Moreover, we
assume that the input vector $y^{[n]}$ satisfies
\begin{equation}\label{VS.input}
y^{[n]}=z_2(x_{n})+{O}(h^{p+1}),
\end{equation}
with $z_2(x_{n}):=\sum_{l=0}^{\rho}\beta_{l}y(x_{n-l})$ for some
vectors $\beta_l =
[\beta_{1,l}\hspace{2mm}\beta_{2,l}\hspace{2mm}\ldots\hspace{2mm}
\beta_{r,l}]^T$, and $h=\max_{0\leq n\leq N-1} h_{n}$. Then, we
request that
\begin{equation}\label{VS.out}
y^{[n+1]}=z_2(x_{n+1})+{O}(h^{p+1}),
\end{equation}
with the same vectors $\beta_l$, $l=0,1,\ldots,\rho$. Now,
denoting the local discretization errors for the stages and output
values respectively by $h_n d_1 (x_{n+1})$ and $h_n d_2
(x_{n+1})$, we have
\begin{equation}\label{Approximate}
\left\{\begin{array}{l}
z_{1}{(x_{n+1})}=h_{n}A(\sigma_{n})f(z_{1}{(x_{n+1})})
+h_{n}^2\overline{A}(\sigma_{n})g(z_{1}{(x_{n+1})})
+U(\sigma_{n})z_{2}{(x_{n})} +h_n d_1 (x_{n+1}), \\[2mm]
z_{2}{(x_{n+1})}=h_{n}B(\sigma_{n})f(z_{1}{(x_{n+1})}) +
h_{n}^2\overline{B}(\sigma_{n})g(z_{1}{(x_{n+1})})
+V(\sigma_{n})z_{2}{(x_{n})}+h_n d_2 (x_{n+1}),
\end{array}
\right.
\end{equation}
or equivalently,
\begin{align}\label{Errors}
 \begin{array}{l}
 \displaystyle
  h_n d_1 (x_{n+1})=y(x_n + c(\sigma_n)h_n)-
  h_n A(\sigma_{n})y'(x_n+c(\sigma_n)h_n)
  -h_{n}^2\overline{A}(\sigma_{n})y''(x_n + c(\sigma_n)h_n)\\[2mm]
   \displaystyle\hspace{20mm}-U(\sigma_{n}) \sum_{l=0}^{\rho}\beta_l
  y(x_{n-l}),\\[2mm]
 \displaystyle
  h_n d_2 (x_{n+1})=\sum_{l=0}^{\rho}\beta_l y(x_{n+1-l})-
  h_n B(\sigma_{n})y'(x_n + c(\sigma_n)h_n)-
h_{n}^2\overline{B}(\sigma_{n})y''(x_n + c(\sigma_n)h_n)\\[3mm]
 \displaystyle\hspace{20mm}-V(\sigma_n) \sum_{l=0}^{\rho}\beta_l y(x_{n-l}).
 \end{array}
 \end{align}
Expanding $y(x_n + c(\sigma_n)h_n)$, $y'(x_n + c(\sigma_n)h_n)$
and $y''(x_n + c(\sigma_n)h_n)$ around the point $x_n$, we get
\begin{equation}\label{C.Ch}
\left\{
  \begin{array}{ll}
    \displaystyle
    h_n d_1 (x_{n+1})=C_0 (\sigma_n)y(x_n)+\sum_{\mu=1}^p h_n^{\mu}C_{\mu} (\sigma_n)y^{(\mu)}(x_n) + O(h^{p+1}),
    \\[3mm]
    \displaystyle h_n d_2 (x_{n+1})=\widehat{C}_0 (\sigma_n)y(x_n)+\sum_{\mu=1}^p h_n^{\mu}\widehat{C}_{\mu} (\sigma_n)y^{(\mu)}(x_n)+O(h^{p+1}),
  \end{array}
\right.
\end{equation}
in which
\begin{equation}\label{OC.C.formula}
\left\{
\begin{array}{ll}
\displaystyle
C_0 (\sigma_n):=e-U(\sigma_{n}) \sum_{l=0}^{\rho}\beta_l, \\[3mm]
\displaystyle
C_1 (\sigma_n):=c(\sigma_n)-A(\sigma_n)e
+U(\sigma_n)\sum_{l=0}^{\rho}\bigl(\sum _{\nu=1}^{l} \sigma_{n,\nu}\bigr)\beta_l,\\[4mm]
\displaystyle C_{\mu}(\sigma_n):=\frac{c(\sigma_n)^{\mu}}{\mu !}
-A(\sigma_n)\frac{c(\sigma_n)^{\mu-1}}{(\mu-1) !}  -\overline
{A}(\sigma_n)\frac{c(\sigma_n)^{\mu-2}}{(\mu-2) !} -
\frac{(-1)^{\mu}}{\mu !}U(\sigma_n)\sum_{l=0}^{\rho}\bigl(\sum
_{\nu=1}^{l} \sigma_{n,\nu}\bigr)^{\mu}\beta_l,
\end{array}
\right.
\end{equation}
and
\begin{equation}\label{OC.Ch.formula}
\left\{
\begin{array}{ll}
\displaystyle
\widehat{C}_0 (\sigma_n):=(I-V(\sigma_{n})) \sum_{l=0}^{\rho}\beta_l, \\[3mm]
\displaystyle \widehat{C}_1(\sigma_n):=\beta_0-B(\sigma_n)e-
\sum_{l=2}^{\rho}\bigl(\sum _{\nu=1}^{l-1}
\sigma_{n,\nu}\bigr)\beta_l+
V(\sigma_{n})\sum_{l=0}^{\rho}\bigl(\sum _{\nu=1}^{l}\sigma_{n,\nu}\bigr)\beta_l,\\[4mm]
\displaystyle \widehat{C}_{\mu}(\sigma_n):=\frac{\beta_0}{\mu
!}-B(\sigma_{n})\frac{c(\sigma_n)^{\mu-1}}{(\mu-1)
!}-\overline{B}(\sigma_n)\frac{c(\sigma_{n})^{\mu-2}}{(\mu-2) !}
+\frac{(-1)^{\mu}}{\mu !}\Bigl(\sum_{l=2}^{\rho}\bigl(\sum
_{\nu=1}^{l-1}
\sigma_{n,\nu}\bigr)^{\mu} \beta_l\\[4mm]
\displaystyle \hspace{20mm}-V(\sigma_n)\sum_{l=1}^{\rho}\bigl(\sum
_{\nu=1}^{l} \sigma_{n,\nu}\bigr)^{\mu}\beta_{l}\Bigr),
\end{array}
\right.
\end{equation}
for $\mu=2,3,\ldots,p$, where $I$ is identity matrix of dimension
equal to $s$.

\begin{definition}
The method \eqref{VS.SGLMs} is said to be preconsistent if $C_0
(\sigma_n)=\widehat{C}_0 (\sigma_n)=0$. Also, it is said to be
consistent if it is preconsistent and $C_1
(\sigma_n)=\widehat{C}_1 (\sigma_n)=0$.
\end{definition}

It should be noted that the preconsistency is equivalent to the
condition $\sum_{l=0}^{\rho}\beta_l=e$. One can find that the
order and stage order of the method \eqref{VS.SGLMs} is $p$ in
which $p$ is the largest positive integer that
$d_1(x_{n+1})=O(h^p)$ and $d_2 (x_{n+1})=O(h^p)$,
$n=0,1,\ldots,N-1$. To state this more precisely, we give the
following theorem.
\begin{theorem}\label{thm.OC}
Assume that $y^{[n]}$ satisfies \eqref{VS.input}. Then, the VS
SDIMSIM \eqref{VS.SGLMs} has order $p$ and stage order $q=p$
satisfying \eqref{VS.stage} and \eqref{VS.out} if and only if
$C_{\mu} (\sigma_n)=\widehat{C}_{\mu}(\sigma_n)=0$ for
$\mu=0,1,\ldots,p$.
\end{theorem}

Let us introduce the matrix
$\beta:=[\beta_0\hspace{3mm}\beta_1\hspace{3mm}\cdots\hspace{3mm}\beta_\rho]$
and the matrices $T(\sigma_n)\in \mathbb{R}^{(\rho+1)\times
(p+1)}$, $C(\sigma_n)\in\mathbb{R}^{s\times(p+1)}$, and
$\widehat{T}(\sigma_n)\in \mathbb{R}^{(\rho+1)\times (p+1)}$ as
\begin{equation*}
T(\sigma_n):=
 \left[
 \begin{array}{cccccc}
  1 & 0 & 0 & \cdots & 0 & 0 \\[2mm]
  1 & -\sigma_{n,1} & \frac{\sigma_{n,1}^2}{2!} & \cdots & \frac{(-1)^{p-1}\sigma_{n,1}^{p-1}}{(p-1)!} & \frac{(-1)^{p}\sigma_{n,1}^{p}}{p!}\\[2mm]
  1 & -(\sigma_{n,1}+\sigma_{n,2}) & \frac{(\sigma_{n,1}+\sigma_{n,2})^2}{2!} & \cdots & \frac{(-1)^{p-1}(\sigma_{n,1}+\sigma_{n,2})^{p-1}}{(p-1)!} &
  \frac{(-1)^{p}(\sigma_{n,1}+\sigma_{n,2})^{p}}{p!}\\[2mm]
  \vdots & \vdots & \vdots & \ddots &\vdots & \vdots \\[2mm]
  \displaystyle
  1 &\displaystyle -\sum_{\nu=1}^{\rho}\sigma_{n,\nu} &\displaystyle
  \frac{(\sum_{\nu=1}^{\rho}\sigma_{n,\nu})^2}{2!} &\displaystyle \cdots &
  \displaystyle \frac{(-1)^{p-1}(\sum_{\nu=1}^{\rho}\sigma_{n,\nu})^{p-1}}{(p-1)!}
  &\displaystyle \frac{(-1)^{p}(\sum_{\nu=1}^{\rho}\sigma_{n,\nu})^{p}}{p!}
 \end{array}
 \right],
\end{equation*}
and
\begin{equation*}
C(\sigma_n):=\left[e\hspace{2mm}{c(\sigma_n)}\hspace{2mm}
\frac{c(\sigma_n)^2}{2!}\hspace{2mm} \cdots\hspace{2mm}
\frac{c(\sigma_n)^p}{p!}\right], \qquad \widehat{T}(\sigma_n):=
 \left[
 \begin{array}{c}
  E_1  \\[2mm]
  \mathbf{T}(\sigma_n)
 \end{array}
 \right],
\end{equation*}
where
$E_1=[1\hspace{2mm}1\hspace{2mm}\frac{1}{2!}\hspace{2mm}\cdots\hspace{2mm}
\frac{1}{p!}]$ and $\mathbf{T}(\sigma_n)$ is the matrix obtained
by the $\rho$ first rows of $T(\sigma_n)$. Now, we state
equivalent conditions for the order conditions in the case $p=q$
in the following corollary.

\begin{corollary}
The order conditions for VS SDIMSIM \eqref{VS.SGLMs} of order $p$
and stage order $q=p$ are
\begin{equation}\label{OC.Matrix}
\begin{aligned}
C(\sigma_n)&=A(\sigma_n)C(\sigma_n)K
+\overline{A}(\sigma_n)C(\sigma_n)K^2+U(\sigma_n)\beta T(\sigma_n), \\[2mm]
\beta \widehat{T}(\sigma_n)&=B(\sigma_n)C(\sigma_n)K
+\overline{B}(\sigma_n)C(\sigma_n)K^2 +V(\sigma_n)\beta
T(\sigma_n).
\end{aligned}
\end{equation}
\end{corollary}

\begin{proof}
Collect $C_{\mu}(\sigma_n)$ and $\widehat{C}_{\mu}(\sigma_n)$,
$\mu=0,1,\ldots,p$ as the columns of two matrices. Being zero
these matrices is equivalent to \eqref{OC.Matrix}.
$\hfill\blacksquare$
\end{proof}

The local truncation error $lte(x_{n+1})$ of the VS method
\eqref{VS.SGLMs} with $q=p$ at the point $x_{n+1}$ is determined
by the following theorem.
\begin{theorem}
The local truncation error $lte(x_{n+1})$ of the method
\eqref{VS.SGLMs} with $q=p$ at the point $x_{n+1}$ is given by
\[
lte(x_{n+1})=\mathcal{\phi}_p
(\sigma_n)h_n^{p+1}y^{(p+1)}(x_n)+O(h_n^{p+2}),
\]
with
\begin{equation}\label{LTE}
\mathcal{\phi}_p (\sigma_n)=\beta
\widehat{T}_{p+1}(\sigma_n)-B(\sigma_n)C_{p}(\sigma_n)
-\overline{B}(\sigma_n)C_{p-1}(\sigma_n)-V(\sigma_n)\beta
T_{p+1}(\sigma_n),
\end{equation}
in which $C_p(\sigma_n)=\dfrac{c(\sigma_n)^p}{p!}$,
$C_{p-1}(\sigma_n)=\dfrac{c(\sigma_n)^{p-1}}{(p-1)!}$,
\[
T_{p+1}(\sigma_n)=\Bigl[0\hspace{3mm}\frac{(-1)^{p+1}\sigma_{n,1})^{p+1}}{(p+1)!}
\hspace{3mm}\cdots\hspace{3mm}
\frac{(-1)^{p+1}(\sum_{\nu=1}^{\rho}\sigma_{n,\nu})^{p+1}}{(p+1)!}\Bigr]^T,
\]
and
\[
\widehat{T}_{p+1}(\sigma_n)=\Bigl[\frac1{(p+1)!}\hspace{3mm}0\hspace{3mm}\frac{(-1)^{p+1}\sigma_{n,1})^{p+1}}{(p+1)!}
\hspace{3mm}\cdots\hspace{3mm}
\frac{(-1)^{p+1}(\sum_{\nu=1}^{\rho-1}\sigma_{n,\nu})^{p+1}}{(p+1)!}\Bigr]^T.
\]
\end{theorem}
\proof By the second relation of \eqref{Approximate}, we have
\begin{align*}
lte(x_{n+1})&=h_nd_2
(x_{n+1})\\
&=z_{2}{(x_{n+1})}-h_{n}B(\sigma_{n})f(z_{1}{(x_{n+1})})-
h_{n}^2\overline{B}(\sigma_{n})g(z_{1}{(x_{n+1})})
-V(\sigma_{n})z_{2}{(x_{n})}.
\end{align*}
Using the definitions for $z_1$ and $z_2$, expanding the terms in
the right hand side of the last relation into Taylor series around
$x_n$, we obtain
\begin{align*}
lte(x_{n+1})&=\Bigl(\beta
\widehat{T}(\sigma_n)-B(\sigma_n)C(\sigma_n)K
-\overline{B}(\sigma_n)C(\sigma_n)K^2-V(\sigma_n)\beta
T(\sigma_n)\Bigr)z(x_n,h_n)\\
&\hspace{4mm}+\Bigl(\beta
\widehat{T}_{p+1}(\sigma_n)-B(\sigma_n)C_{p}(\sigma_n)
-\overline{B}(\sigma_n)C_{p-1}(\sigma_n)-V(\sigma_n)\beta
T_{p+1}(\sigma_n)\Bigr)h_n^{p+1}y^{(p+1)}(x_n)\\
&\hspace{4mm}+O(h_n^{p+2}).
\end{align*}
Now, considering the order conditions for the method of order $p$
and stage order $q=p$ in \eqref{OC.Matrix} completes the proof.
$\hfill\blacksquare$
\setcounter{equation}{0} \setcounter{definition}{0}
\setcounter{theorem}{0}
\section{Construction of VS explicit SDIMSIMs}\label{Sec3}
In this section, we are going to construct VS explicit SDIMSIMs
with $p=q=r=s\leq 4$. The coefficients matrices of such methods
are given as a function of the ratios of stepsizes. Throughout the
paper, we consider the abscissa vector $c$ to be values uniformly
in the interval $[0,1]$ so that
$c=[0\hspace{2mm}\frac{1}{s-1}\hspace{2mm}\ldots\hspace{2mm}\frac{s-2}{s-1}\hspace{2mm}1]^T$.
Construction of the VS SDIMSIMs can be done by many approaches.
Here, we assume $\rho=p-1$ and $\beta=I$. It should be noted that
in this way, the input and out put vectors take the form
\begin{align*}
y^{[n]}=\left[%
\begin{array}{c}
  y_n \\
  y_{n-1} \\
  \vdots \\
  y_{n-p+1} \\
\end{array}%
\right],\qquad
y^{[n+1]}=\left[%
\begin{array}{c}
  y_{n+1} \\
  y_{n} \\
  \vdots \\
  y_{n-p+2} \\
\end{array}%
\right].
\end{align*}
After applying the order and stage order conditions
\eqref{OC.Matrix}, we obtain methods with some free parameters
which are chosen in such a way that the underlying fixed stepsize
SDIMSIM has a large stability region with a small error constant
$\mathcal{C}_p$.

\subsection{Methods of order $p=1$}\label{Subsec.3.1}
In this subsection, we derive methods of order one. Since, in this
case $y^{[n]}=y_n$ and $y^{[n+1]}=y_{n+1}$, the coefficients of
the methods are independent of $\sigma_n$. Applying the order and
stage order conditions, the coefficients of the methods take the
form
\begin{align*}
\left[
\begin{array}{c|c|c}
 A & \overline{A} & U\\
 \hline
 &&\\[-4mm]
 B & \overline{B} & V\\
\end{array}
\right]= \left[
\begin{array}{c|c|c}
 0  & 0  &  1\\
 \hline
 &&\\[-4mm]
 1 & \overline{b} & 1\\
\end{array}
\right].
\end{align*}
The stability function of the method is
$R(z)=1+z+\overline{b}z^2$; the condition $\mathcal{C}_1=10^{-3}$
leads to $\overline{b}=\frac{499}{1000}$.
\subsection{Methods of order $p=2$}\label{Subsec.3.2}
In this subsection, we are going to obtain methods of order two
with the abscissa vector $c=[0\hspace{2mm}1]^T$. The coefficients
matrices of these methods take the form
\begin{eqnarray*}
\left[
\begin{array}{c|c|c}
 A(\sigma_{n}) & \overline{A}(\sigma_{n}) & U(\sigma_{n})\\
 \hline
  &&\\[-4mm]
 B(\sigma_{n}) & \overline{B}(\sigma_{n}) & V(\sigma_{n})\\
\end{array}\right]
= \left[
\begin{array}{cc|cc|cc}
 0 & 0 & 0 & 0 & u_{11}(\sigma_{n}) & 1-u_{11}(\sigma_{n})\\
 a_{21} (\sigma_{n})& 0 & \overline{a}_{21}(\sigma_{n}) & 0 & u_{21}(\sigma_{n}) & 1-u_{21}(\sigma_{n})\\
 \hline
 &&&&&\\[-4mm]
 b_{11}(\sigma_{n}) & b_{12} (\sigma_{n})& \overline{b}_{11}(\sigma_{n}) & \overline{b}_{12}(\sigma_{n}) & v_{1}(\sigma_{n}) & 1-v_{1}(\sigma_{n})\\
 b_{21}(\sigma_{n}) & b_{22} (\sigma_{n})& \overline{b}_{21}(\sigma_{n}) & \overline{b}_{22}(\sigma_{n}) & v_{1}(\sigma_{n}) & 1-v_{1}(\sigma_{n})\\
\end{array}
\right],
\end{eqnarray*}
with $\sigma_n=\sigma_{n,1}$. Solving the order and stage order
conditions \eqref{OC.Matrix} leads to
\begin{align*}
a_{21}(\sigma_{n})&=(\sigma_{n}+1-
2\overline{a}_{21}(\sigma_{n}))/\sigma_{n},\\[2mm]
u_{11}(\sigma_{n})&=1,\\[2mm]
u_{21}(\sigma_{n})&=(-1+2\overline{a}_{21}(\sigma_{n})+
\sigma_{n}^2)/\sigma_{n}^2,\\[2mm]
b_{11}(\sigma_{n})&= \frac12+\overline{b}_{11}(\sigma_{n})+
\overline{b}_{12}(\sigma_{n})+\frac12\sigma_{n}^2
-\frac12\sigma_{n}^2v_1(\sigma_{n})+\sigma_{n}-
\sigma_{n}v_1(\sigma_{n}),\\[2mm]
b_{12}(\sigma_{n})&=
\frac12-\overline{b}_{11}(\sigma_{n})-\overline{b}_{12}(\sigma_{n})-
\frac12\sigma_{n}^2+\frac12\sigma_{n}^2v_1(\sigma_{n}),\\[2mm]
b_{21}(\sigma_{n})&=
\overline{b}_{21}(\sigma_{n})+\overline{b}_{22}(\sigma_{n})+
\frac12\sigma_{n}^2-\frac12\sigma_{n}^2v_1(\sigma_{n})+\sigma_{n}-
\sigma_{n}v_1(\sigma_{n}),\\[2mm]
b_{22}(\sigma_{n})&=
-\overline{b}_{21}(\sigma_{n})-\overline{b}_{22}(\sigma_{n})-
\frac12\sigma_{n}^2+\frac12\sigma_{n}^2v_1(\sigma_{n}),
\end{align*}
in which $\overline{a}_{21}(\sigma_{n})$,
$\overline{b}_{11}(\sigma_{n})$, $\overline{b}_{12}(\sigma_{n})$,
$\overline{b}_{21}(\sigma_{n})$, $\overline{b}_{22}(\sigma_{n})$,
and $v_{1}(\sigma_{n})$ are free parameters. Now, defining
\[
M(z;\sigma_n):=V+\bigl(zB+z^2\overline{B}\bigr)(I-zA-z^2\overline{A})^{-1}U,
\]
and
\[
p(w,z;\sigma_n):=\det(wI-M(z;\sigma_n)),
\]
as generalizations of the stability matrix and stability function
of the method \eqref{SGLMs} (cf. \cite{AbdHoj2011a,ButHoj2005}),
for these methods, we have
\[
p(w,z;\sigma_n)=w^2+p_1(z;\sigma_n)w+p_0(z;\sigma_n),
\]
in which $p_1$ and $p_0$ are polynomials in terms of $z$ of degree
four. By requiring $p_0(z;\sigma_n)\equiv0$ which leads to
Runge--Kutta stable methods \cite{AbdHoj2011a,ButHoj2005} in the
case of fixed stepsize and imposing that the underlying fixed
stepsize method has a large stability region with the error
constant $\mathcal{C}_2=10^{-3}$, the coefficients matrices of the
method take the form
\begin{align*}
A(\sigma_n)&=\left[%
\begin{array}{cc}
  0 & 0 \\[2mm]
  1+\dfrac1{5\sigma_n} & 0 \\
\end{array}%
\right],\qquad
\overline{A}(\sigma_n)=\left[%
\begin{array}{cc}
  0 & 0 \\[2mm]
  \dfrac25 & 0 \\
\end{array}%
\right],\qquad
U(\sigma_n)=\left[%
\begin{array}{cc}
  1 & 0 \\[2mm]
  1-\dfrac1{5\sigma_n^2} & \dfrac1{5\sigma_n^2} \\
\end{array}%
\right],\\[3mm]
B(\sigma_n)&=\left[%
\begin{array}{cc}
  \dfrac34+\dfrac{253}{4500}\sigma_n & \dfrac14 \\[3mm]
  -\dfrac14+\dfrac{253}{4500}\sigma_n+\dfrac{253}{900}\sigma_n^2 &
  \dfrac14-\dfrac{253}{900}\sigma_n^2 \\
\end{array}%
\right],\\[3mm]
\overline{B}(\sigma_n)&=\left[%
\begin{array}{cc}
  \dfrac18+\dfrac{253}{6000}\sigma_n^2 &
  \dfrac18-\dfrac{253}{3600}\sigma_n^2 \\[3mm]
  -\dfrac18+\dfrac{3289}{18000}\sigma_n^2 &
  -\dfrac18+\dfrac{253}{3600}\sigma_n^2 \\
\end{array}%
\right],\qquad
V(\sigma_n)=\left[%
\begin{array}{cc}
  \dfrac{4247}{4500} & \dfrac{253}{4500} \\[3mm]
  \dfrac{4247}{4500} & \dfrac{253}{4500} \\
\end{array}%
\right].
\end{align*}
\subsection{Methods of order $p=3$}\label{Subsec.3.3}
In this subsection, we derive methods of order three with the
abscissa vector $c=[0\hspace{2mm}\frac{1}{2}\hspace{2mm}1]^T$. The
coefficients matrices of these methods have the following form
\begin{eqnarray*}
\left[
\begin{array}{ccc|ccc|ccc}
 0 & 0 & 0 & 0 & 0 & 0 & u_{11}(\sigma_{n}) & u_{12}(\sigma_{n}) & 1-u_{11}(\sigma_{n})-u_{12}(\sigma_{n})\\
 a_{21} (\sigma_{n})& 0 & 0 & \overline{a}_{21}(\sigma_{n}) & 0 & 0 & u_{21}(\sigma_{n}) & u_{22}(\sigma_{n}) & 1-u_{21}(\sigma_{n})-u_{22}(\sigma_{n})\\
 a_{31} (\sigma_{n})& a_{32} (\sigma_{n}) & 0 & \overline{a}_{31}(\sigma_{n}) & \overline{a}_{32}(\sigma_{n}) & 0 & u_{31}(\sigma_{n}) & u_{32}(\sigma_{n}) & 1-u_{31}(\sigma_{n})-u_{32}(\sigma_{n})\\
 \hline\\[-3mm]
 b_{11}(\sigma_{n}) & b_{12} (\sigma_{n}) & b_{13} (\sigma_{n}) & \overline{b}_{11}(\sigma_{n}) & \overline{b}_{12}(\sigma_{n}) & \overline{b}_{13}(\sigma_{n}) & v_{1}(\sigma_{n}) & v_{2}(\sigma_{n}) &  1-v_{1}(\sigma_{n})-v_{2}(\sigma_{n})\\
 b_{21}(\sigma_{n}) & b_{22} (\sigma_{n}) & b_{23} (\sigma_{n}) & \overline{b}_{21}(\sigma_{n}) & \overline{b}_{22}(\sigma_{n}) & \overline{b}_{23}(\sigma_{n}) & v_{1}(\sigma_{n}) & v_{2}(\sigma_{n}) &  1-v_{1}(\sigma_{n})-v_{2}(\sigma_{n})\\
 b_{31}(\sigma_{n}) & b_{32} (\sigma_{n}) & b_{33} (\sigma_{n}) & \overline{b}_{31}(\sigma_{n}) & \overline{b}_{32}(\sigma_{n}) & \overline{b}_{33}(\sigma_{n}) & v_{1}(\sigma_{n}) & v_{2}(\sigma_{n}) &  1-v_{1}(\sigma_{n})-v_{2}(\sigma_{n})\\
\end{array}
\right],
\end{eqnarray*}
with $\sigma_n=[\sigma_{n,1}\hspace{3mm}\sigma_{n,2}]^T$. To
construct such methods, after applying the order and stage order
conditions  \eqref{OC.Matrix} for entries of the matrices
$U(\sigma_n)$, $B(\sigma_n)$ and the first column of the matrix
$A(\sigma_n)$, the remain parameters are chosen in such a way that
the underlying fixed stepsize method has a large stability region
with `nice' coefficients and the error constant
$\mathcal{C}_3=10^{-3}$. The coefficients matrices of the method
take the form
\begin{align*}
A(\sigma_n)&=
\left[%
\begin{array}{ccc}
  0 & 0 & 0 \\[2mm]
  \frac{5+2\sigma_{n,2}+20\sigma_{n,1}\sigma_{n,2}+20\sigma_{n,1}^2+4\sigma_{n,1}}{40\sigma_{n,1}(\sigma_{n,1}+\sigma_{n,2})} & 0 & 0 \\[3mm]
  \frac{-55-52\sigma_{n,2}+60\sigma_{n,1}\sigma_{n,2}+60\sigma_{n,1}^2-104\sigma_{n,1}}{80\sigma_{n,1}(\sigma_{n,1}+\sigma_{n,2})} & \frac14 & 0 \\
\end{array}%
\right],\quad \overline{A}(\sigma_n)=
\left[%
\begin{array}{ccc}
  0 & 0 & 0 \\[2mm]
  \frac1{10} & 0 & 0 \\[2mm]
  \frac15 & \frac12 & 0 \\
\end{array}%
\right],\\[3mm]
\overline{B}(\sigma_n)&=\left[%
\begin{array}{ccc}
  \frac{67}{500} & 0 & \frac{13}{500} \\[2mm]
  0 & -\frac{171}{500} & 0 \\[2mm]
  -\frac{321}{100} & 0 & -\frac{73}{100} \\
\end{array}%
\right],\qquad
V(\sigma_n)=\left[%
\begin{array}{ccc}
  0 & \frac{12072}{9889} & -\frac{2183}{9889} \\[2mm]
  0 & \frac{12072}{9889} & -\frac{2183}{9889} \\[2mm]
  0 & \frac{12072}{9889} & -\frac{2183}{9889} \\
\end{array}%
\right],\\[3mm]
U(\sigma_n)&=\left[%
\begin{array}{ccc}
  1 & 0 & 0 \\[2mm]
  u_{21}(\sigma_n) &
  \frac{5+2(\sigma_{n,1}+\sigma_{n,2})}{40\sigma_{n,1}^2\sigma_{n,2}} &
  -\frac{5+2\sigma_{n,1}}{40\sigma_{n,2}(\sigma_{n,1}+\sigma_{n,2})^2}
  \\[3mm]
  u_{31}(\sigma_n)&
  -\frac{55+52(\sigma_{n,1}+\sigma_{n,2})}{80\sigma_{n,1}^2\sigma_{n,2}} &
  \frac{55+52\sigma_{n,1}}{80\sigma_{n,2}(\sigma_{n,1}+\sigma_{n,2})^2}
  \\
\end{array}%
\right],
\end{align*}
with
\begin{align*}
u_{21}(\sigma_n)&=\frac{-6\sigma_{n,1}\sigma_{n,2}-5\sigma_{n,2}-10\sigma_{n,1}-6\sigma_{n,1}^2-2\sigma_{n,2}^2+40\sigma_{n,1}^2\sigma_{n,2}^2+80\sigma_{n,1}^3\sigma_{n,2}+40\sigma_{n,1}^4}{40\sigma_{n,1}^2(\sigma_{n,1}+\sigma_{n,2})^2},\\[2mm]
u_{31}(\sigma_n)&=\frac{156\sigma_{n,1}\sigma_{n,2}+55\sigma_{n,2}+110\sigma_{n,1}+156\sigma_{n,1}^2+52\sigma_{n,2}^2+80\sigma_{n,1}^2\sigma_{n,2}^2+160\sigma_{n,1}^3\sigma_{n,2}+80\sigma_{n,1}^4}{80\sigma_{n,1}^2(\sigma_{n,1}+\sigma_{n,2})^2},
\end{align*}
and the elements of the matrix $B$ are given by
\begin{align*}
b_{11}(\sigma_n)&=\sigma_{n,1}-{\frac
{2183}{9889}}\,\sigma_{n,2}+\frac32\,{\sigma_{n,1}}^{2
}+\frac23\,{\sigma_{n,1}}^{3}-{\frac
{6549}{9889}}\,\sigma_{n,1}\sigma_{{2 }}+{\frac {407}{750}}-{\frac
{4366}{29667}}\,{\sigma_{n,2}}^{3}\\[2mm]
&\hspace{4mm}-{ \frac
{4366}{9889}}\,\sigma_{n,1}{\sigma_{n,2}}^{2}-{\frac {4366}{9889
}}\,{\sigma_{n,1}}^{2}\sigma_{n,2}-{\frac
{6549}{19778}}\,{\sigma_{n,2}}^{2},
\\[2mm]
b_{12}(\sigma_n)&={\frac {88}{375}}+{\frac
{8732}{9889}}\,\sigma_{n,1}{\sigma_{n,2}}^{2} +{\frac
{8732}{9889}}\,\sigma_{n,1}\sigma_{n,2}+{\frac {4366}{9889}}\,
{\sigma_{n,2}}^{2}-\frac43\,{\sigma_{n,1}}^{3}-2\,{\sigma_{n,1}}^{2}\\[2mm]
&\hspace{4mm}+{ \frac
{8732}{9889}}\,{\sigma_{n,1}}^{2}\sigma_{n,2}+{\frac {8732}{
29667}}\,{\sigma_{n,2}}^{3},
\\[2mm]
b_{13}(\sigma_n)&={\frac {167}{750}}-{\frac
{4366}{9889}}\,\sigma_{n,1}{\sigma_{n,2}}^{2 }-{\frac
{2183}{9889}}\,\sigma_{n,1}\sigma_{n,2}-{\frac {2183}{19778}}
\,{\sigma_{n,2}}^{2}+\frac23\,{\sigma_{n,1}}^{3}
+\frac12\,{\sigma_{n,1}}^{2}\\[2mm]
&\hspace{4mm}-{ \frac
{4366}{9889}}\,{\sigma_{n,1}}^{2}\sigma_{n,2}-{\frac {4366}{
29667}}\,{\sigma_{n,2}}^{3},
\\[2mm]
b_{21}(\sigma_n)&=-{\frac {171}{500}}+\sigma_{n,1}-{\frac
{2183}{9889}}\,\sigma_{n,2}+
\frac32\,{\sigma_{n,1}}^{2}+\frac23\,{\sigma_{n,1}}^{3}-{\frac
{6549}{9889}}\, \sigma_{n,1}\sigma_{n,2}-{\frac
{4366}{29667}}\,{\sigma_{n,2}}^{3}\\[2mm]
&\hspace{4mm}-{ \frac
{4366}{9889}}\,\sigma_{n,1}{\sigma_{n,2}}^{2}-{\frac {4366}{9889
}}\,{\sigma_{n,1}}^{2}\sigma_{n,2}-{\frac
{6549}{19778}}\,{\sigma_{n,2}}^{2},
\\[2mm]
b_{22}(\sigma_n)&={\frac
{8732}{9889}}\,\sigma_{n,1}{\sigma_{n,2}}^{2}+{\frac {8732}{
9889}}\,\sigma_{n,1}\sigma_{n,2}+{\frac
{4366}{9889}}\,{\sigma_{n,2}}^
{2}-\frac43\,{\sigma_{n,1}}^{3}-2\,{\sigma_{n,1}}^{2}\\[2mm]
&\hspace{4mm}+{\frac {8732}{9889}}
\,{\sigma_{n,1}}^{2}\sigma_{n,2}+{\frac
{8732}{29667}}\,{\sigma_{n,2}} ^{3},
\\[2mm]
b_{23}(\sigma_n)&={\frac {171}{500}}-{\frac
{4366}{9889}}\,\sigma_{n,1}{\sigma_{n,2}}^{2 }-{\frac
{2183}{9889}}\,\sigma_{n,1}\sigma_{n,2}-{\frac {2183}{19778}}
\,{\sigma_{n,2}}^{2}+\frac23\,{\sigma_{n,1}}^{3}
+\frac12\,{\sigma_{n,1}}^{2}\\[2mm]
&\hspace{4mm}-{ \frac
{4366}{9889}}\,{\sigma_{n,1}}^{2}\sigma_{n,2}-{\frac {4366}{
29667}}\,{\sigma_{n,2}}^{3},
\\[2mm]
b_{31}(\sigma_n)&=-{\frac {89}{10}}-{\frac
{4366}{9889}}\,\sigma_{n,1}{\sigma_{n,2}}^{2} -{\frac
{6549}{9889}}\,\sigma_{n,1}\sigma_{n,2}-{\frac {2183}{9889}}\,
\sigma_{n,2}-{\frac {6549}{19778}}\,{\sigma_{n,2}}^{2}\\[2mm]
&\hspace{4mm}-{\frac {4366}{
9889}}\,{\sigma_{n,1}}^{2}\sigma_{n,2}-{\frac
{4366}{29667}}\,{\sigma_ {{2}}}^{3},
\\[2mm]
b_{32}(\sigma_n)&={\frac {248}{25}}+{\frac
{8732}{9889}}\,\sigma_{n,1}{\sigma_{n,2}}^{2} +{\frac
{8732}{9889}}\,\sigma_{n,1}\sigma_{n,2}+{\frac {4366}{9889}}\,
{\sigma_{n,2}}^{2}+{\frac
{8732}{9889}}\,{\sigma_{n,1}}^{2}\sigma_{n,2}+{\frac
{8732}{29667}}\,{\sigma_{n,2}}^{3},
\\[2mm]
b_{33}(\sigma_n)&=-{\frac {51}{50}}-{\frac
{2183}{9889}}\,\sigma_{n,1}\sigma_{n,2}-{ \frac
{2183}{19778}}\,{\sigma_{n,2}}^{2}-{\frac {4366}{9889}}\,{\sigma
_{{1}}}^{2}\sigma_{n,2}-{\frac
{4366}{9889}}\,\sigma_{n,1}{\sigma_{n,2}}^{2}-{\frac
{4366}{29667}}\,{\sigma_{n,2}}^{3}.
\\[2mm]
\end{align*}
\subsection{Methods of order $p=4$}\label{Subsec.3.4}
In this subsection, we search for explicit methods of order four
with the abscissa vector
$c=[0\hspace{2mm}\frac{1}{3}\hspace{2mm}\frac{2}{3}\hspace{2mm}1]^T$
in which the matrices $A(\sigma_{n})$ and
$\overline{A}(\sigma_{n})$ are strictly lower triangular matrices,
the matrix $U$ has the rows
\[
u_i(\sigma_{n})=\Bigl[u_{11}(\sigma_{n})\hspace{3mm}u_{12}(\sigma_{n})
\hspace{3mm}u_{13}(\sigma_{n})\hspace{3mm}1-\sum_{j=1}^{3}u_{ij}(\sigma_n)\Bigr],
\quad
i=1,2,3,4,
\]
and the matrix $V$ is a rank-one matrix in the form
$V(\sigma_{n})=e\,v(\sigma_{n})^T$ where
$v(\sigma_{n})=[v_1(\sigma_{n})\hspace{2mm}v_2(\sigma_{n})\hspace{2mm}\ldots\hspace{2mm}v_r(\sigma_{n})]^T$,
and $v(\sigma_{n})^Te=1$ with with
$\sigma_n=[\sigma_{n,1}\hspace{3mm}\sigma_{n,2}\hspace{3mm}\sigma_{n,3}]^T$.
By applying the order and stage order conditions \eqref{OC.Matrix}
for entries of the matrices $U(\sigma_n)$, $B(\sigma_n)$ and the
first column of the matrix $A(\sigma_n)$, the remain parameters
are chosen so that the underlying fixed stepsize method has a
large stability region with `nice' coefficients and the error
constant $\mathcal{C}_3\approx2\times10^{-3}$. The coefficients
matrices of the method take the form
\begin{align*}
A(\sigma_n)&=\left[%
\begin{array}{cccc}
  0 & 0 & 0 & 0 \\[2mm]
  a_{21}(\sigma_n) & 0 & 0 & 0 \\[2mm]
  a_{31}(\sigma_n) & -\frac{11}{25} & 0 & 0 \\[2mm]
  a_{41}(\sigma_n) & \frac{11}{10} & -\frac{16}{25} & 0 \\
\end{array}%
\right],\qquad
\overline{A}(\sigma_n)&=\left[%
\begin{array}{cccc}
  0 & 0 & 0 & 0 \\[2mm]
  \frac12 & 0 & 0 & 0 \\[2mm]
  1 & \frac14 & 0 & 0 \\[2mm]
  \frac{351}{125} & 0 & \frac{42}{125} & 0 \\
\end{array}%
\right],
\end{align*}
with {\footnotesize
\begin{align*} a_{21}(\sigma_n)&={\frac
{1}{81D(\sigma_n)}}\,\Bigl(-72\,{\sigma_{n,2}}^{2}+27\,\sigma_{n,1}{
\sigma_{n,2}}^{2}-72\,\sigma_{n,2}\sigma_{n,3}+6\,\sigma_{n,2}+54\,{
\sigma_{n,1}}^{2}\sigma_{n,2}+27\,\sigma_{n,1}\sigma_{n,2}\sigma_{n,3}
\\[1mm]
&\hspace{4mm}-288\,\sigma_{n,1}\sigma_{n,2}-216\,{\sigma_{n,1}}^{2}
+27\,{\sigma_{{n,1
}}}^{3}+3\,\sigma_{n,3}-144\,\sigma_{n,1}\sigma_{n,3}+27\,{\sigma_{n,1
}}^{2}\sigma_{n,3}+9\,\sigma_{n,1}+1\Bigr),
\\[1mm]
a_{31}(\sigma_n)&={\frac
{1}{4050D(\sigma_n)}}\,\Bigl(-7137\,{\sigma_{n,2}}^{2}+4482\,\sigma_{n,1}{\sigma_{n,2}}^{2}-7137\,
\sigma_{n,2}\sigma_{n,3}-462\,\sigma_{n,2}+8964\,{\sigma_{n,1}}^{2}
\sigma_{n,2}\\[1mm]
&\hspace{4mm}+4482\,\sigma_{n,1}\sigma_{n,2}\sigma_{n,3}-28548\,\sigma_
{{1}}\sigma_{n,2}-21411\,{\sigma_{n,1}}^{2}+4482\,{\sigma_{n,1}}^{3}-
231\,\sigma_{n,3}-14274\,\sigma_{n,1}\sigma_{n,3}\\[1mm]
&\hspace{4mm}+4482\,{\sigma_{n,1}}
^{2}\sigma_{n,3}-693\,\sigma_{n,1}-286
\Bigr)\\[1mm]
a_{41}(\sigma_n)&={\frac
{1}{2250D(\sigma_n)}}\,\Bigl(1215\,\sigma_{n,1}{\sigma_{n,2}}^{2}-11628\,{\sigma_{n,2}}^{2}+1215\,
\sigma_{n,1}\sigma_{n,2}\sigma_{n,3}+642\,\sigma_{n,2}+2430\,{\sigma_{
{1}}}^{2}\sigma_{n,2}\\[1mm]
&\hspace{4mm}-11628\,\sigma_{n,2}\sigma_{n,3}-46512\,\sigma_{{
1}}\sigma_{n,2}-34884\,{\sigma_{n,1}}^{2}+1215\,{\sigma_{n,1}}^{3}+321
\,\sigma_{n,3}+1215\,{\sigma_{n,1}}^{2}\sigma_{n,3}\\[1mm]
&\hspace{4mm}-23256\,\sigma_{n,1
}\sigma_{n,3}+963\,\sigma_{n,1}-442 \Bigr),
\end{align*}}
in which $D(\sigma_n)$ stands for $\sigma_{n,1}(\sigma_{n,1}
+\sigma_{n,2}) (\sigma_{n,1}+\sigma_{n,2}+\sigma_{n,3})$, and
\begin{align*}
V(\sigma_n)&=\left[ \begin {array}{cccc} \frac12&\frac14&{\frac
{8}{25}}&-{\frac {7}{100}}
\\\noalign{\medskip}\frac12&\frac14&{\frac {8}{25}}&-{\frac {7}{100}}
\\\noalign{\medskip}\frac12&\frac14&{\frac {8}{25}}&-{\frac {7}{100}}
\\\noalign{\medskip}\frac12&\frac14&{\frac {8}{25}}&-{\frac {7}{100}}
\end {array} \right],
\qquad \overline{B}(\sigma_n)=\left[ \begin {array}{cccc} {\frac
{6211}{25000}}&{\frac {2}{25}}&-{ \frac
{147}{6250}}&0\\\noalign{\medskip}{\frac {6211}{25000}}&{\frac {
2}{25}}&-{\frac {147}{6250}}&0\\\noalign{\medskip}{\frac
{6211}{25000} }&{\frac {2}{25}}&-{\frac
{147}{6250}}&0\\\noalign{\medskip}{\frac { 6211}{25000}}&{\frac
{2}{25}}&-{\frac {147}{6250}}&0\end {array}
 \right]
\end{align*}
and {\footnotesize
\begin{align*}
u_{1}(\sigma_n)&=\bigl[\,1\hspace{3mm}0\hspace{3mm}0\hspace{3mm}0\,\bigr],\\[2mm]
u_{21}(\sigma_n)&=\frac1{81D(\sigma_n)^2}\Bigl(-3\,{\sigma_{n,1}}^{2}-24\,{\sigma_{n,1}}^{3}+432\,{\sigma_{n,1}}^{4}-
2\,\sigma_{n,1}\sigma_{n,3}-24\,{\sigma_{n,1}}^{2}\sigma_{n,3}-30\,
\sigma_{n,1}\sigma_{n,2}\sigma_{n,3}+216\,\sigma_{n,1}\sigma_{n,2}{
\sigma_{n,3}}^{2}\\[1mm]
&\hspace{4mm}+1080\,{\sigma_{n,1}}^{2}\sigma_{n,2}\sigma_{n,3}+81
\,{\sigma_{n,2}}^{2}{\sigma_{n,3}}^{2}{\sigma_{n,1}}^{2}+162\,{\sigma_
{{2}}}^{3}\sigma_{n,3}{\sigma_{n,1}}^{2}+162\,\sigma_{n,2}{\sigma_{n,3
}}^{2}{\sigma_{n,1}}^{3}+486\,{\sigma_{n,1}}^{3}{\sigma_{n,2}}^{2}
\sigma_{n,3}\\[1mm]
&\hspace{4mm}+486\,{\sigma_{n,1}}^{4}\sigma_{n,2}\sigma_{n,3}+648\,
\sigma_{n,1}{\sigma_{n,2}}^{2}\sigma_{n,3}-30\,\sigma_{n,1}{\sigma_{{n,2
}}}^{2}-{\sigma_{n,2}}^{2}-\sigma_{n,2}\sigma_{n,3}-48\,{\sigma_{n,1}}
^{2}\sigma_{n,2}-4\,\sigma_{n,1}\sigma_{n,2}-6\,{\sigma_{n,2}}^{3}\\[1mm]
&\hspace{4mm}+72
\,{\sigma_{n,2}}^{4}+81\,{\sigma_{n,1}}^{6}+1152\,{\sigma_{n,1}}^{3}
\sigma_{n,2}+1080\,{\sigma_{n,1}}^{2}{\sigma_{n,2}}^{2}+432\,\sigma_{{
1}}{\sigma_{n,2}}^{3}+576\,{\sigma_{n,1}}^{3}\sigma_{n,3}+216\,{\sigma
_{{1}}}^{2}{\sigma_{n,3}}^{2}\\[1mm]
&\hspace{4mm}+72\,{\sigma_{n,2}}^{2}{\sigma_{n,3}}^{2}
+144\,{\sigma_{n,2}}^{3}\sigma_{n,3}+486\,{\sigma_{n,1}}^{4}{\sigma_{{
2}}}^{2}+162\,{\sigma_{n,1}}^{5}\sigma_{n,3}+81\,{\sigma_{n,1}}^{4}{
\sigma_{n,3}}^{2}+324\,{\sigma_{n,1}}^{5}\sigma_{n,2}+81\,{\sigma_{n,2
}}^{4}{\sigma_{n,1}}^{2}\\[1mm]
&\hspace{4mm}+324\,{\sigma_{n,1}}^{3}{\sigma_{n,2}}^{3}-9\,
{\sigma_{n,2}}^{2}\sigma_{n,3}-3\,\sigma_{n,2}{\sigma_{n,3}}^{2}-6\,
\sigma_{n,1}{\sigma_{n,3}}^{2}
\Bigr),\\[1mm]
u_{22}(\sigma_n)&=\frac{-1}{81(\sigma_{n,2}+
\sigma_{n,3})\sigma_{n,1}^2\sigma_{n,2}}\Bigl(72\,\sigma_{n,2}\sigma_{n,3}
-3\,\sigma_{n,3}+72\,\sigma_{n,1}\sigma_{
{3}}+72\,{\sigma_{n,2}}^{2}-6\,\sigma_{n,2}+144\,\sigma_{n,1}\sigma_{{
2}}-6\,\sigma_{n,1}-1+72\,{\sigma_{n,1}}^{2}
\Bigr),\\[1mm]
u_{23}(\sigma_n)&=\frac1{81(\sigma_{n,1}+\sigma_{n,2})^2\sigma_{n,2}\sigma_{n,3}}\Bigl(-3\,\sigma_{n,3}+72\,\sigma_{n,1}\sigma_{n,3}-1+72\,{\sigma_{n,1}}^{2}
-3\,\sigma_{n,2}-6\,\sigma_{n,1}+72\,\sigma_{n,1}\sigma_{n,2}
\Bigr),\\[1mm]
u_{24}(\sigma_n)&=\frac{-1}{81\sigma_{n,3}(\sigma_{n,2}
+\sigma_{n,3})
(\sigma_{n,1}+\sigma_{n,2}+\sigma_{n,3})^2}\Bigl(72\,{\sigma_{n,1}}^{2}-6\,\sigma_{n,1}+72\,\sigma_{n,1}\sigma_{n,2}-1-
3\,\sigma_{n,2} \Bigr),
\\[1mm]
u_{31}(\sigma_n)&=\frac1{4050D(\sigma_n)^2}\Bigl(2310\,\sigma_{n,1}\sigma_{n,2}\sigma_{n,3}+8100\,{\sigma_{n,1}}^{2}{
\sigma_{n,2}}^{3}\sigma_{n,3}+4050\,{\sigma_{n,1}}^{2}{\sigma_{n,3}}^{
2}{\sigma_{n,2}}^{2}+107055\,{\sigma_{n,1}}^{2}\sigma_{n,2}\sigma_{n,3
}+2310\,\sigma_{n,1}{\sigma_{n,2}}^{2}\\[1mm]
&\hspace{4mm}+286\,\sigma_{n,2}\sigma_{n,3}+
1144\,\sigma_{n,1}\sigma_{n,2}+3696\,{\sigma_{n,1}}^{2}\sigma_{n,2}+
572\,\sigma_{n,1}\sigma_{n,3}+1848\,{\sigma_{n,1}}^{2}\sigma_{n,3}+231
\,\sigma_{n,2}{\sigma_{n,3}}^{2}+693\,{\sigma_{n,2}}^{2}\sigma_{n,3}\\[1mm]
&\hspace{4mm}+
462\,\sigma_{n,1}{\sigma_{n,3}}^{2}+114192\,{\sigma_{n,1}}^{3}\sigma_{
{2}}+107055\,{\sigma_{n,1}}^{2}{\sigma_{n,2}}^{2}+57096\,{\sigma_{n,1}
}^{3}\sigma_{n,3}+21411\,{\sigma_{n,1}}^{2}{\sigma_{n,3}}^{2}+14274\,{
\sigma_{n,2}}^{3}\sigma_{n,3}\\[2mm]
&\hspace{4mm}+7137\,{\sigma_{n,2}}^{2}{\sigma_{n,3}}^{
2}+42822\,\sigma_{n,1}{\sigma_{n,2}}^{3}+4050\,{\sigma_{n,1}}^{4}{
\sigma_{n,3}}^{2}+4050\,{\sigma_{n,1}}^{2}{\sigma_{n,2}}^{4}+16200\,{
\sigma_{n,1}}^{3}{\sigma_{n,2}}^{3}+24300\,{\sigma_{n,1}}^{4}{\sigma_{
{2}}}^{2}\\[1mm]
&\hspace{4mm}+8100\,{\sigma_{n,1}}^{5}\sigma_{n,3}+16200\,{\sigma_{n,1}}^{
5}\sigma_{n,2}+21411\,\sigma_{n,1}\sigma_{n,2}{\sigma_{n,3}}^{2}+24300
\,{\sigma_{n,1}}^{3}{\sigma_{n,2}}^{2}\sigma_{n,3}+858\,{\sigma_{n,1}}
^{2}+1848\,{\sigma_{n,1}}^{3}+42822\,{\sigma_{n,1}}^{4}\\[1mm]
&\hspace{4mm}+24300\,{\sigma
_{{1}}}^{4}\sigma_{n,2}\sigma_{n,3}+64233\,\sigma_{n,1}{\sigma_{n,2}}^
{2}\sigma_{n,3}+8100\,{\sigma_{n,1}}^{3}{\sigma_{n,3}}^{2}\sigma_{n,2}
+286\,{\sigma_{n,2}}^{2}+462\,{\sigma_{n,2}}^{3}+7137\,{\sigma_{n,2}}^
{4}+4050\,{\sigma_{n,1}}^{6}
\Bigr),\\[2mm]
u_{32}(\sigma_n)&=\frac{-1}{4050(\sigma_{n,2}+
\sigma_{n,3})\sigma_{n,1}^2\sigma_{n,2}}\Bigl(7137\,\sigma_{n,2}\sigma_{n,3}+231\,\sigma_{n,3}+7137\,\sigma_{n,1}
\sigma_{n,3}+7137\,{\sigma_{n,2}}^{2}+462\,\sigma_{n,2}+14274\,\sigma_
{{1}}\sigma_{n,2}+286\\[1mm]
&\hspace{4mm}+7137\,{\sigma_{n,1}}^{2}+462\,\sigma_{n,1}
\Bigr),\\[2mm]
u_{33}(\sigma_n)&=\frac1{4050(\sigma_{n,1}+\sigma_{n,2})^2\sigma_{n,2}\sigma_{n,3}}\Bigl(231\,\sigma_{n,3}+7137\,\sigma_{n,1}\sigma_{n,3}+7137\,\sigma_{n,1}
\sigma_{n,2}+231\,\sigma_{n,2}+462\,\sigma_{n,1}+286+7137\,{\sigma_{{n,1
}}}^{2}
\Bigr),\\[2mm]
u_{34}(\sigma_n)&=\frac{-1}{4050\sigma_{n,3}(\sigma_{n,2}
+\sigma_{n,3})
(\sigma_{n,1}+\sigma_{n,2}+\sigma_{n,3})^2}\Bigl(7137\,\sigma_{n,1}\sigma_{n,2}+231\,\sigma_{n,2}+462\,\sigma_{n,1}+286
+7137\,{\sigma_{n,1}}^{2}
\Bigr),\\[2mm]
u_{41}(\sigma_n)&=\frac1{2250D(\sigma_n)^2}\Bigl(-3210\,\sigma_{n,1}\sigma_{n,2}\sigma_{n,3}+4500\,{\sigma_{n,1}}^{2}{
\sigma_{n,2}}^{3}\sigma_{n,3}+2250\,{\sigma_{n,1}}^{2}{\sigma_{n,3}}^{
2}{\sigma_{n,2}}^{2}+174420\,{\sigma_{n,1}}^{2}\sigma_{n,2}\sigma_{n,3
}-3210\,\sigma_{n,1}{\sigma_{n,2}}^{2}\\[1mm]
&\hspace{4mm}+442\,\sigma_{n,2}\sigma_{n,3}+
1768\,\sigma_{n,1}\sigma_{n,2}-5136\,{\sigma_{n,1}}^{2}\sigma_{n,2}+
884\,\sigma_{n,1}\sigma_{n,3}-2568\,{\sigma_{n,1}}^{2}\sigma_{n,3}-321
\,\sigma_{n,2}{\sigma_{n,3}}^{2}-963\,{\sigma_{n,2}}^{2}\sigma_{n,3}\\[1mm]
&\hspace{4mm}-
642\,\sigma_{n,1}{\sigma_{n,3}}^{2}+186048\,{\sigma_{n,1}}^{3}\sigma_{
{2}}+174420\,{\sigma_{n,1}}^{2}{\sigma_{n,2}}^{2}+93024\,{\sigma_{n,1}
}^{3}\sigma_{n,3}+34884\,{\sigma_{n,1}}^{2}{\sigma_{n,3}}^{2}+23256\,{
\sigma_{n,2}}^{3}\sigma_{n,3}\\[1mm]
&\hspace{4mm}+11628\,{\sigma_{n,2}}^{2}{\sigma_{n,3}}^
{2}+69768\,\sigma_{n,1}{\sigma_{n,2}}^{3}+2250\,{\sigma_{n,1}}^{4}{
\sigma_{n,3}}^{2}+2250\,{\sigma_{n,1}}^{2}{\sigma_{n,2}}^{4}+9000\,{
\sigma_{n,1}}^{3}{\sigma_{n,2}}^{3}+13500\,{\sigma_{n,1}}^{4}{\sigma_{
{2}}}^{2}\\[2mm]
&\hspace{4mm}+4500\,{\sigma_{n,1}}^{5}\sigma_{n,3}+9000\,{\sigma_{n,1}}^{5
}\sigma_{n,2}+34884\,\sigma_{n,1}\sigma_{n,2}{\sigma_{n,3}}^{2}+13500
\,{\sigma_{n,1}}^{3}{\sigma_{n,2}}^{2}\sigma_{n,3}+1326\,{\sigma_{n,1}
}^{2}-2568\,{\sigma_{n,1}}^{3}+69768\,{\sigma_{n,1}}^{4}\\[1mm]
&\hspace{4mm}+13500\,{
\sigma_{n,1}}^{4}\sigma_{n,2}\sigma_{n,3}+104652\,\sigma_{n,1}{\sigma_
{{2}}}^{2}\sigma_{n,3}+4500\,{\sigma_{n,1}}^{3}{\sigma_{n,3}}^{2}
\sigma_{n,2}+442\,{\sigma_{n,2}}^{2}-642\,{\sigma_{n,2}}^{3}+11628\,{
\sigma_{n,2}}^{4}+2250\,{\sigma_{n,1}}^{6} \Bigr),\end{align*}}
{\footnotesize
\begin{align*}
u_{42}(\sigma_n)&=\frac{-1}{2250(\sigma_{n,2}+
\sigma_{n,3})\sigma_{n,1}^2\sigma_{n,2}}\Bigl(11628\,\sigma_{n,2}\sigma_{n,3}-321\,\sigma_{n,3}+11628\,\sigma_{n,1}
\sigma_{n,3}+11628\,{\sigma_{n,2}}^{2}-642\,\sigma_{n,2}+23256\,\sigma
_{{1}}\sigma_{n,2}\\[1mm]
&\hspace{4mm}-642\,\sigma_{n,1}+11628\,{\sigma_{n,1}}^{2}+442
\Bigr),\\[1mm]
u_{43}(\sigma_n)&=\frac1{2250(\sigma_{n,1}+\sigma_{n,2})^2\sigma_{n,2}\sigma_{n,3}}\Bigl(-321\,\sigma_{n,3}+11628\,\sigma_{n,1}\sigma_{n,3}+11628\,\sigma_{n,1}
\sigma_{n,2}-321\,\sigma_{n,2}-642\,\sigma_{n,1}+442+11628\,{\sigma_{{
1}}}^{2}
\Bigr),\\[1mm]
u_{44}(\sigma_n)&=\frac{-1}{2250\sigma_{n,3}(\sigma_{n,2}
+\sigma_{n,3})
(\sigma_{n,1}+\sigma_{n,2}+\sigma_{n,3})^2}\Bigl(11628\,\sigma_{n,1}\sigma_{n,2}-321\,\sigma_{n,2}-642\,\sigma_{n,1}+
442+11628\,{\sigma_{n,1}}^{2} \Bigr),\\[1mm]
b_{11}(\sigma_n)&=-{\frac
{63}{50}}\,\sigma_{n,1}\sigma_{n,2}\sigma_{n,3}-{\frac {189}{
200}}\,{\sigma_{n,1}}^{2}\sigma_{n,2}\sigma_{n,3}+\frac94\,\sigma_{n,1}{
\sigma_{n,2}}^{2}{\frac
{77}{200}}\,\sigma_{n,2}\sigma_{n,3}+{\frac {
11}{8}}\,\sigma_{n,1}\sigma_{n,2}+\frac94\,{\sigma_{n,1}}^{2}\sigma_{n,2}-
{\frac {77}{200}}\,\sigma_{n,1}\sigma_{n,3}\\[1mm]
&\hspace{4mm}-{\frac {63}{100}}\,{\sigma
_{{1}}}^{2}\sigma_{n,3}-{\frac
{63}{100}}\,\sigma_{n,2}{\sigma_{n,3}}^ {2}-{\frac
{63}{100}}\,{\sigma_{n,2}}^{2}\sigma_{n,3}-{\frac {63}{100}
}\,\sigma_{n,1}{\sigma_{n,3}}^{2}+{\frac
{9}{8}}\,{\sigma_{n,1}}^{3} \sigma_{n,2}+{\frac
{27}{16}}\,{\sigma_{n,1}}^{2}{\sigma_{n,2}}^{2}-{ \frac
{63}{200}}\,{\sigma_{n,1}}^{3}\sigma_{n,3}\\[1mm]
&\hspace{4mm}-{\frac {189}{400}}\,{
\sigma_{n,1}}^{2}{\sigma_{n,3}}^{2}-{\frac
{63}{200}}\,{\sigma_{n,2}}^ {3}\sigma_{n,3}-{\frac
{189}{400}}\,{\sigma_{n,2}}^{2}{\sigma_{n,3}}^{ 2}+{\frac
{9}{8}}\,\sigma_{n,1}{\sigma_{n,2}}^{3}-{\frac {189}{200}}\,
\sigma_{n,1}\sigma_{n,2}{\sigma_{n,3}}^{2}-{\frac
{77}{400}}\,{\sigma_ {{3}}}^{2}-{\frac
{63}{200}}\,\sigma_{n,1}{\sigma_{n,3}}^{3}\\[1mm]
&\hspace{4mm}+\frac12\, \sigma_{n,1}+\frac14\,\sigma_{n,2}-{\frac
{7}{100}}\,\sigma_{n,3}+{\frac {
11}{8}}\,{\sigma_{n,1}}^{2}+\frac32\,{\sigma_{n,1}}^{3}+{\frac
{9}{16}}\,{ \sigma_{n,1}}^{4}-{\frac
{63}{800}}\,{\sigma_{n,3}}^{4}-{\frac {189}{
200}}\,\sigma_{n,1}{\sigma_{n,2}}^{2}\sigma_{n,3}-{\frac
{63}{200}}\, \sigma_{n,2}{\sigma_{n,3}}^{3}\\[1mm]
&\hspace{4mm}-{\frac {21}{100}}\,{\sigma_{n,3}}^{3}+{ \frac
{11}{16}}\,{\sigma_{n,2}}^{2}+\frac34\,{\sigma_{n,2}}^{3}+{\frac
{9} {32}}\,{\sigma_{n,2}}^{4}+{\frac {79159}{50000}}
,\\[1mm]
b_{12}(\sigma_n)&={\frac
{63}{20}}\,\sigma_{n,1}\sigma_{n,2}\sigma_{n,3}+{\frac {567}{
200}}\,{\sigma_{n,1}}^{2}\sigma_{n,2}\sigma_{n,3}-{\frac
{45}{8}}\, \sigma_{n,1}{\sigma_{n,2}}^{2}+{\frac
{63}{100}}\,\sigma_{n,2}\sigma_{
{3}}-\frac94\,\sigma_{n,1}\sigma_{n,2}-{\frac
{45}{8}}\,{\sigma_{n,1}}^{2} \sigma_{n,2}+{\frac
{63}{100}}\,\sigma_{n,1}\sigma_{n,3}\\[1mm]
&\hspace{4mm}+{\frac {63}{
40}}\,{\sigma_{n,1}}^{2}\sigma_{n,3}+{\frac
{63}{40}}\,\sigma_{n,2}{ \sigma_{n,3}}^{2}+{\frac
{63}{40}}\,{\sigma_{n,2}}^{2}\sigma_{n,3}+{ \frac
{63}{40}}\,\sigma_{n,1}{\sigma_{n,3}}^{2}-{\frac {27}{8}}\,{
\sigma_{n,1}}^{3}\sigma_{n,2}-{\frac
{81}{16}}\,{\sigma_{n,1}}^{2}{ \sigma_{n,2}}^{2}+{\frac
{189}{200}}\,{\sigma_{n,1}}^{3}\sigma_{n,3}\\[1mm]
&\hspace{4mm}+{ \frac
{567}{400}}\,{\sigma_{n,1}}^{2}{\sigma_{n,3}}^{2}+{\frac {189}{
200}}\,{\sigma_{n,2}}^{3}\sigma_{n,3}+{\frac
{567}{400}}\,{\sigma_{n,2 }}^{2}{\sigma_{n,3}}^{2}-{\frac
{27}{8}}\,\sigma_{n,1}{\sigma_{n,2}}^{ 3}+{\frac
{567}{200}}\,\sigma_{n,1}\sigma_{n,2}{\sigma_{n,3}}^{2}+{ \frac
{63}{200}}\,{\sigma_{n,3}}^{2}+{\frac {189}{200}}\,\sigma_{n,1}{
\sigma_{n,3}}^{3}\\[1mm]
&\hspace{4mm}-{\frac
{5661}{3125}}-\frac94\,{\sigma_{n,1}}^{2}-{\frac
{15}{4}}\,{\sigma_{n,1}}^{3}-{\frac
{27}{16}}\,{\sigma_{n,1}}^{4}+{ \frac
{189}{800}}\,{\sigma_{n,3}}^{4}+{\frac {567}{200}}\,\sigma_{n,1}
{\sigma_{n,2}}^{2}\sigma_{n,3}+{\frac
{189}{200}}\,\sigma_{n,2}{\sigma _{{3}}}^{3}+{\frac
{21}{40}}\,{\sigma_{n,3}}^{3}-{\frac {9}{8}}\,{
\sigma_{n,2}}^{2}\\[1mm]
&\hspace{4mm}-{\frac {15}{8}}\,{\sigma_{n,2}}^{3}-{\frac {27}{32}}
\,{\sigma_{n,2}}^{4}
,\\[1mm]
b_{13}(\sigma_n)&=-{\frac
{63}{25}}\,\sigma_{n,1}\sigma_{n,2}\sigma_{n,3}-{\frac {567}{
200}}\,{\sigma_{n,1}}^{2}\sigma_{n,2}\sigma_{n,3}+\frac92\,\sigma_{n,1}{
\sigma_{n,2}}^{2}-{\frac
{63}{200}}\,\sigma_{n,2}\sigma_{n,3}+{\frac {
9}{8}}\,\sigma_{n,1}\sigma_{n,2}+\frac92\,{\sigma_{n,1}}^{2}\sigma_{n,2}-{
\frac {63}{200}}\,\sigma_{n,1}\sigma_{n,3}\\[1mm]
&\hspace{4mm}-{\frac {63}{50}}\,{\sigma_{
{1}}}^{2}\sigma_{n,3}-{\frac
{63}{50}}\,\sigma_{n,2}{\sigma_{n,3}}^{2} -{\frac
{63}{50}}\,{\sigma_{n,2}}^{2}\sigma_{n,3}-{\frac {63}{50}}\,
\sigma_{n,1}{\sigma_{n,3}}^{2}+{\frac {27}{8}}\,{\sigma_{n,1}}^{3}
\sigma_{n,2}+{\frac
{81}{16}}\,{\sigma_{n,1}}^{2}{\sigma_{n,2}}^{2}-{ \frac
{189}{200}}\,{\sigma_{n,1}}^{3}\sigma_{n,3}\\[1mm]
&\hspace{4mm}-{\frac {567}{400}}\,
{\sigma_{n,1}}^{2}{\sigma_{n,3}}^{2}-{\frac
{189}{200}}\,{\sigma_{n,2} }^{3}\sigma_{n,3}-{\frac
{567}{400}}\,{\sigma_{n,2}}^{2}{\sigma_{n,3}} ^{2}+{\frac
{27}{8}}\,\sigma_{n,1}{\sigma_{n,2}}^{3}-{\frac {567}{200}
}\,\sigma_{n,1}\sigma_{n,2}{\sigma_{n,3}}^{2}-{\frac {63}{400}}\,{
\sigma_{n,3}}^{2}-{\frac
{189}{200}}\,\sigma_{n,1}{\sigma_{n,3}}^{3}\\[1mm]
&\hspace{4mm}+{ \frac
{9}{8}}\,{\sigma_{n,1}}^{2}+3\,{\sigma_{n,1}}^{3}+{\frac {27}{16
}}\,{\sigma_{n,1}}^{4}-{\frac
{189}{800}}\,{\sigma_{n,3}}^{4}+{\frac { 64413}{50000}}-{\frac
{567}{200}}\,\sigma_{n,1}{\sigma_{n,2}}^{2} \sigma_{n,3}-{\frac
{189}{200}}\,\sigma_{n,2}{\sigma_{n,3}}^{3}-{ \frac
{21}{50}}\,{\sigma_{n,3}}^{3}+{\frac {9}{16}}\,{\sigma_{n,2}}^{2
}\\[1mm]
&\hspace{4mm}+\frac32\,{\sigma_{n,2}}^{3}+{\frac
{27}{32}}\,{\sigma_{n,2}}^{4}
,\\[1mm]
b_{14}(\sigma_n)&={\frac
{63}{100}}\,\sigma_{n,1}\sigma_{n,2}\sigma_{n,3}+{\frac {189}{
200}}\,{\sigma_{n,1}}^{2}\sigma_{n,2}\sigma_{n,3}-{\frac {9}{8}}\,
\sigma_{n,1}{\sigma_{n,2}}^{2}+{\frac
{7}{100}}\,\sigma_{n,2}\sigma_{{
3}}-\frac14\,\sigma_{n,1}\sigma_{n,2}-{\frac
{9}{8}}\,{\sigma_{n,1}}^{2} \sigma_{n,2}+{\frac
{7}{100}}\,\sigma_{n,1}\sigma_{n,3}\\[1mm]
&\hspace{4mm}+{\frac {63}{
200}}\,{\sigma_{n,1}}^{2}\sigma_{n,3}+{\frac
{63}{200}}\,\sigma_{n,2}{ \sigma_{n,3}}^{2}+{\frac
{63}{200}}\,{\sigma_{n,2}}^{2}\sigma_{n,3}+{ \frac
{63}{200}}\,\sigma_{n,1}{\sigma_{n,3}}^{2}-{\frac {9}{8}}\,{
\sigma_{n,1}}^{3}\sigma_{n,2}-{\frac
{27}{16}}\,{\sigma_{n,1}}^{2}{ \sigma_{n,2}}^{2}+{\frac
{63}{200}}\,{\sigma_{n,1}}^{3}\sigma_{n,3}\\[1mm]
&\hspace{4mm}+{ \frac
{189}{400}}\,{\sigma_{n,1}}^{2}{\sigma_{n,3}}^{2}+{\frac {63}{
200}}\,{\sigma_{n,2}}^{3}\sigma_{n,3}+{\frac
{189}{400}}\,{\sigma_{n,2 }}^{2}{\sigma_{n,3}}^{2}-{\frac
{9}{8}}\,\sigma_{n,1}{\sigma_{n,2}}^{3 }+{\frac
{189}{200}}\,\sigma_{n,1}\sigma_{n,2}{\sigma_{n,3}}^{2}+{ \frac
{7}{200}}\,{\sigma_{n,3}}^{2}+{\frac {63}{200}}\,\sigma_{n,1}{
\sigma_{n,3}}^{3}\\[1mm]
&\hspace{4mm}-\frac14\,{\sigma_{n,1}}^{2}-\frac34\,{\sigma_{n,1}}^{3}-{
\frac {9}{16}}\,{\sigma_{n,1}}^{4}+{\frac
{63}{800}}\,{\sigma_{n,3}}^{ 4}-{\frac {749}{12500}}+{\frac
{189}{200}}\,\sigma_{n,1}{\sigma_{n,2}} ^{2}\sigma_{n,3}+{\frac
{63}{200}}\,\sigma_{n,2}{\sigma_{n,3}}^{3}+{ \frac
{21}{200}}\,{\sigma_{n,3}}^{3}-\frac18\,{\sigma_{n,2}}^{2}\\[1mm]
&\hspace{4mm}-\frac38\,{ \sigma_{n,2}}^{3}-{\frac
{9}{32}}\,{\sigma_{n,2}}^{4},
\end{align*}}
{\footnotesize
\begin{align*}
b_{21}(\sigma_n)&=-{\frac
{63}{50}}\,\sigma_{n,1}\sigma_{n,2}\sigma_{n,3}-{\frac {189}{
200}}\,{\sigma_{n,1}}^{2}\sigma_{n,2}\sigma_{n,3}+\frac94\,\sigma_{n,1}{
\sigma_{n,2}}^{2}-{\frac
{77}{200}}\,\sigma_{n,2}\sigma_{n,3}+{\frac {
11}{8}}\,\sigma_{n,1}\sigma_{n,2}+\frac94\,{\sigma_{n,1}}^{2}\sigma_{n,2}-
{\frac {77}{200}}\,\sigma_{n,1}\sigma_{n,3}\\[1mm]
&\hspace{4mm}-{\frac {63}{100}}\,{\sigma
_{{1}}}^{2}\sigma_{n,3}-{\frac
{63}{100}}\,\sigma_{n,2}{\sigma_{n,3}}^ {2}-{\frac
{63}{100}}\,{\sigma_{n,2}}^{2}\sigma_{n,3}-{\frac {63}{100}
}\,\sigma_{n,1}{\sigma_{n,3}}^{2}+{\frac
{9}{8}}\,{\sigma_{n,1}}^{3} \sigma_{n,2}+{\frac
{27}{16}}\,{\sigma_{n,1}}^{2}{\sigma_{n,2}}^{2}-{ \frac
{63}{200}}\,{\sigma_{n,1}}^{3}\sigma_{n,3}\\[1mm]
&\hspace{4mm}-{\frac {189}{400}}\,{
\sigma_{n,1}}^{2}{\sigma_{n,3}}^{2}-{\frac
{63}{200}}\,{\sigma_{n,2}}^ {3}\sigma_{n,3}-{\frac
{189}{400}}\,{\sigma_{n,2}}^{2}{\sigma_{n,3}}^{ 2}+{\frac
{9}{8}}\,\sigma_{n,1}{\sigma_{n,2}}^{3}-{\frac {189}{200}}\,
\sigma_{n,1}\sigma_{n,2}{\sigma_{n,3}}^{2}-{\frac
{77}{400}}\,{\sigma_ {{3}}}^{2}-{\frac
{63}{200}}\,\sigma_{n,1}{\sigma_{n,3}}^{3}\\[1mm]
&\hspace{4mm}+{\frac {
72909}{50000}}+\frac12\,\sigma_{n,1}+\frac14\,\sigma_{n,2}-{\frac
{7}{100}}\, \sigma_{n,3}+{\frac
{11}{8}}\,{\sigma_{n,1}}^{2}+\frac32\,{\sigma_{n,1}}^{ 3}+{\frac
{9}{16}}\,{\sigma_{n,1}}^{4}-{\frac {63}{800}}\,{\sigma_{n,3
}}^{4}-{\frac
{189}{200}}\,\sigma_{n,1}{\sigma_{n,2}}^{2}\sigma_{n,3}\\[1mm]
&\hspace{4mm}- {\frac
{63}{200}}\,\sigma_{n,2}{\sigma_{n,3}}^{3}-{\frac {21}{100}}\,{
\sigma_{n,3}}^{3}+{\frac
{11}{16}}\,{\sigma_{n,2}}^{2}+3/4\,{\sigma_{{ 2}}}^{3}+{\frac
{9}{32}}\,{\sigma_{n,2}}^{4}
,\\[1mm]
b_{22}(\sigma_n)&={\frac
{63}{20}}\,\sigma_{n,1}\sigma_{n,2}\sigma_{n,3}+{\frac {567}{
200}}\,{\sigma_{n,1}}^{2}\sigma_{n,2}\sigma_{n,3}-{\frac
{45}{8}}\, \sigma_{n,1}{\sigma_{n,2}}^{2}+{\frac
{63}{100}}\,\sigma_{n,2}\sigma_{
{3}}-9/4\,\sigma_{n,1}\sigma_{n,2}-{\frac
{45}{8}}\,{\sigma_{n,1}}^{2} \sigma_{n,2}+{\frac
{63}{100}}\,\sigma_{n,1}\sigma_{n,3}\\[1mm]
&\hspace{4mm}+{\frac {63}{
40}}\,{\sigma_{n,1}}^{2}\sigma_{n,3}+{\frac
{63}{40}}\,\sigma_{n,2}{ \sigma_{n,3}}^{2}+{\frac
{63}{40}}\,{\sigma_{n,2}}^{2}\sigma_{n,3}+{ \frac
{63}{40}}\,\sigma_{n,1}{\sigma_{n,3}}^{2}-{\frac {27}{8}}\,{
\sigma_{n,1}}^{3}\sigma_{n,2}-{\frac
{81}{16}}\,{\sigma_{n,1}}^{2}{ \sigma_{n,2}}^{2}+{\frac
{189}{200}}\,{\sigma_{n,1}}^{3}\sigma_{n,3}\\[1mm]
&\hspace{4mm}+{ \frac
{567}{400}}\,{\sigma_{n,1}}^{2}{\sigma_{n,3}}^{2}+{\frac {189}{
200}}\,{\sigma_{n,2}}^{3}\sigma_{n,3}+{\frac
{567}{400}}\,{\sigma_{n,2 }}^{2}{\sigma_{n,3}}^{2}-{\frac
{27}{8}}\,\sigma_{n,1}{\sigma_{n,2}}^{ 3}+{\frac
{567}{200}}\,\sigma_{n,1}\sigma_{n,2}{\sigma_{n,3}}^{2}-{ \frac
{54663}{25000}}+{\frac {63}{200}}\,{\sigma_{n,3}}^{2}\\[1mm]
&\hspace{4mm}+{\frac {
189}{200}}\,\sigma_{n,1}{\sigma_{n,3}}^{3}-\frac94\,{\sigma_{n,1}}^{2}-{
\frac {15}{4}}\,{\sigma_{n,1}}^{3}-{\frac
{27}{16}}\,{\sigma_{n,1}}^{4 }+{\frac
{189}{800}}\,{\sigma_{n,3}}^{4}+{\frac {567}{200}}\,\sigma_{{
1}}{\sigma_{n,2}}^{2}\sigma_{n,3}+{\frac
{189}{200}}\,\sigma_{n,2}{ \sigma_{n,3}}^{3}+{\frac
{21}{40}}\,{\sigma_{n,3}}^{3}\\[1mm]
&\hspace{4mm}-{\frac {9}{8}} \,{\sigma_{n,2}}^{2}-{\frac
{15}{8}}\,{\sigma_{n,2}}^{3}-{\frac {27}{ 32}}\,{\sigma_{n,2}}^{4}
,\\[1mm]
b_{23}(\sigma_n)&=-{\frac
{63}{25}}\,\sigma_{n,1}\sigma_{n,2}\sigma_{n,3}-{\frac {567}{
200}}\,{\sigma_{n,1}}^{2}\sigma_{n,2}\sigma_{n,3}+{\frac
{45663}{50000 }}+\frac92\,\sigma_{n,1}{\sigma_{n,2}}^{2}-{\frac
{63}{200}}\,\sigma_{n,2} \sigma_{n,3}+{\frac
{9}{8}}\,\sigma_{n,1}\sigma_{n,2}+9/2\,{\sigma_{{n,1
}}}^{2}\sigma_{n,2}\\[1mm]
&\hspace{4mm}-{\frac {63}{200}}\,\sigma_{n,1}\sigma_{n,3}-{ \frac
{63}{50}}\,{\sigma_{n,1}}^{2}\sigma_{n,3}-{\frac {63}{50}}\,
\sigma_{n,2}{\sigma_{n,3}}^{2}-{\frac
{63}{50}}\,{\sigma_{n,2}}^{2} \sigma_{n,3}-{\frac
{63}{50}}\,\sigma_{n,1}{\sigma_{n,3}}^{2}+{\frac {
27}{8}}\,{\sigma_{n,1}}^{3}\sigma_{n,2}+{\frac
{81}{16}}\,{\sigma_{n,1 }}^{2}{\sigma_{n,2}}^{2}\\[1mm]
&\hspace{4mm}-{\frac {189}{200}}\,{\sigma_{n,1}}^{3}\sigma_
{{3}}-{\frac
{567}{400}}\,{\sigma_{n,1}}^{2}{\sigma_{n,3}}^{2}-{\frac
{189}{200}}\,{\sigma_{n,2}}^{3}\sigma_{n,3}-{\frac {567}{400}}\,{
\sigma_{n,2}}^{2}{\sigma_{n,3}}^{2}+{\frac {27}{8}}\,\sigma_{n,1}{
\sigma_{n,2}}^{3}-{\frac
{567}{200}}\,\sigma_{n,1}\sigma_{n,2}{\sigma_ {{3}}}^{2}-{\frac
{63}{400}}\,{\sigma_{n,3}}^{2}\\[1mm]
&\hspace{4mm}-{\frac {189}{200}}\,
\sigma_{n,1}{\sigma_{n,3}}^{3}+{\frac
{9}{8}}\,{\sigma_{n,1}}^{2}+3\,{ \sigma_{n,1}}^{3}+{\frac
{27}{16}}\,{\sigma_{n,1}}^{4}-{\frac {189}{
800}}\,{\sigma_{n,3}}^{4}-{\frac
{567}{200}}\,\sigma_{n,1}{\sigma_{n,2 }}^{2}\sigma_{n,3}-{\frac
{189}{200}}\,\sigma_{n,2}{\sigma_{n,3}}^{3}- {\frac
{21}{50}}\,{\sigma_{n,3}}^{3}+{\frac {9}{16}}\,{\sigma_{n,2}}^{
2}\\[1mm]
&\hspace{4mm}+\frac32\,{\sigma_{n,2}}^{3}+{\frac
{27}{32}}\,{\sigma_{n,2}}^{4}
,\\[2mm]
b_{24}(\sigma_n)&={\frac
{63}{100}}\,\sigma_{n,1}\sigma_{n,2}\sigma_{n,3}-{\frac {4623}{
25000}}+{\frac
{189}{200}}\,{\sigma_{n,1}}^{2}\sigma_{n,2}\sigma_{n,3} -{\frac
{9}{8}}\,\sigma_{n,1}{\sigma_{n,2}}^{2}+{\frac {7}{100}}\,
\sigma_{n,2}\sigma_{n,3}-1/4\,\sigma_{n,1}\sigma_{n,2}-{\frac
{9}{8}} \,{\sigma_{n,1}}^{2}\sigma_{n,2}+{\frac
{7}{100}}\,\sigma_{n,1}\sigma_ {{3}}\\[1mm]
&\hspace{4mm}+{\frac
{63}{200}}\,{\sigma_{n,1}}^{2}\sigma_{n,3}+{\frac {63}{
200}}\,\sigma_{n,2}{\sigma_{n,3}}^{2}+{\frac
{63}{200}}\,{\sigma_{n,2} }^{2}\sigma_{n,3}+{\frac
{63}{200}}\,\sigma_{n,1}{\sigma_{n,3}}^{2}-{ \frac
{9}{8}}\,{\sigma_{n,1}}^{3}\sigma_{n,2}-{\frac {27}{16}}\,{
\sigma_{n,1}}^{2}{\sigma_{n,2}}^{2}+{\frac
{63}{200}}\,{\sigma_{n,1}}^ {3}\sigma_{n,3}\\[1mm]
&\hspace{4mm}+{\frac
{189}{400}}\,{\sigma_{n,1}}^{2}{\sigma_{n,3}}^{ 2}+{\frac
{63}{200}}\,{\sigma_{n,2}}^{3}\sigma_{n,3}+{\frac {189}{400}
}\,{\sigma_{n,2}}^{2}{\sigma_{n,3}}^{2}-{\frac
{9}{8}}\,\sigma_{n,1}{ \sigma_{n,2}}^{3}+{\frac
{189}{200}}\,\sigma_{n,1}\sigma_{n,2}{\sigma_ {{3}}}^{2}+{\frac
{7}{200}}\,{\sigma_{n,3}}^{2}+{\frac {63}{200}}\,
\sigma_{n,1}{\sigma_{n,3}}^{3}\\[1mm]
&\hspace{4mm}-\frac14\,{\sigma_{n,1}}^{2}-\frac34\,{\sigma_{{
1}}}^{3}-{\frac {9}{16}}\,{\sigma_{n,1}}^{4}+{\frac {63}{800}}\,{
\sigma_{n,3}}^{4}+{\frac
{189}{200}}\,\sigma_{n,1}{\sigma_{n,2}}^{2} \sigma_{n,3}+{\frac
{63}{200}}\,\sigma_{n,2}{\sigma_{n,3}}^{3}+{\frac
{21}{200}}\,{\sigma_{n,3}}^{3}-1/8\,{\sigma_{n,2}}^{2}-3/8\,{\sigma_{{
2}}}^{3}-{\frac {9}{32}}\,{\sigma_{n,2}}^{4}
,\\[2mm]
b_{31}(\sigma_n)&=-{\frac
{63}{50}}\,\sigma_{n,1}\sigma_{n,2}\sigma_{n,3}-{\frac {189}{
200}}\,{\sigma_{n,1}}^{2}\sigma_{n,2}\sigma_{n,3}+\frac94\,\sigma_{n,1}{
\sigma_{n,2}}^{2}-{\frac
{77}{200}}\,\sigma_{n,2}\sigma_{n,3}+{\frac {
11}{8}}\,\sigma_{n,1}\sigma_{n,2}+\frac94\,{\sigma_{n,1}}^{2}\sigma_{n,2}-
{\frac {77}{200}}\,\sigma_{n,1}\sigma_{n,3}\\[1mm]
&\hspace{4mm}-{\frac {63}{100}}\,{\sigma
_{{1}}}^{2}\sigma_{n,3}-{\frac
{63}{100}}\,\sigma_{n,2}{\sigma_{n,3}}^ {2}-{\frac
{63}{100}}\,{\sigma_{n,2}}^{2}\sigma_{n,3}-{\frac {63}{100}
}\,\sigma_{n,1}{\sigma_{n,3}}^{2}+{\frac
{9}{8}}\,{\sigma_{n,1}}^{3} \sigma_{n,2}+{\frac
{27}{16}}\,{\sigma_{n,1}}^{2}{\sigma_{n,2}}^{2}-{ \frac
{63}{200}}\,{\sigma_{n,1}}^{3}\sigma_{n,3}\\[1mm]
&\hspace{4mm}-{\frac {189}{400}}\,{
\sigma_{n,1}}^{2}{\sigma_{n,3}}^{2}-{\frac
{63}{200}}\,{\sigma_{n,2}}^ {3}\sigma_{n,3}-{\frac
{189}{400}}\,{\sigma_{n,2}}^{2}{\sigma_{n,3}}^{ 2}+{\frac
{9}{8}}\,\sigma_{n,1}{\sigma_{n,2}}^{3}-{\frac {189}{200}}\,
\sigma_{n,1}\sigma_{n,2}{\sigma_{n,3}}^{2}-{\frac
{77}{400}}\,{\sigma_ {{3}}}^{2}-{\frac
{63}{200}}\,\sigma_{n,1}{\sigma_{n,3}}^{3}\\[1mm]
&\hspace{4mm}+{\frac {
72909}{50000}}-\frac12\,\sigma_{n,1}+\frac14\,\sigma_{n,2}-{\frac
{7}{100}}\, \sigma_{n,3}-{\frac
{11}{8}}\,{\sigma_{n,1}}^{2}-\frac32\,{\sigma_{n,1}}^{ 3}-{\frac
{9}{16}}\,{\sigma_{n,1}}^{4}-{\frac {63}{800}}\,{\sigma_{n,3
}}^{4}-{\frac
{189}{200}}\,\sigma_{n,1}{\sigma_{n,2}}^{2}\sigma_{n,3}- {\frac
{63}{200}}\,\sigma_{n,2}{\sigma_{n,3}}^{3}\\[1mm]
&\hspace{4mm}-{\frac {21}{100}}\,{ \sigma_{n,3}}^{3}+{\frac
{11}{16}}\,{\sigma_{n,2}}^{2}+\frac34\,{\sigma_{{ 2}}}^{3}+{\frac
{9}{32}}\,{\sigma_{n,2}}^{4},
\end{align*}}
{\footnotesize
\begin{align*}
b_{32}(\sigma_n)&={\frac
{63}{20}}\,\sigma_{n,1}\sigma_{n,2}\sigma_{n,3}+{\frac {567}{
200}}\,{\sigma_{n,1}}^{2}\sigma_{n,2}\sigma_{n,3}-{\frac
{45}{8}}\, \sigma_{n,1}{\sigma_{n,2}}^{2}+{\frac
{63}{100}}\,\sigma_{n,2}\sigma_{
{3}}-\frac94\,\sigma_{n,1}\sigma_{n,2}-{\frac
{45}{8}}\,{\sigma_{n,1}}^{2} \sigma_{n,2}+{\frac
{63}{100}}\,\sigma_{n,1}\sigma_{n,3}\\[1mm]
&\hspace{4mm}+{\frac {63}{
40}}\,{\sigma_{n,1}}^{2}\sigma_{n,3}+{\frac
{63}{40}}\,\sigma_{n,2}{ \sigma_{n,3}}^{2}+{\frac
{63}{40}}\,{\sigma_{n,2}}^{2}\sigma_{n,3}+{ \frac
{63}{40}}\,\sigma_{n,1}{\sigma_{n,3}}^{2}-{\frac {27}{8}}\,{
\sigma_{n,1}}^{3}\sigma_{n,2}-{\frac
{81}{16}}\,{\sigma_{n,1}}^{2}{ \sigma_{n,2}}^{2}+{\frac
{189}{200}}\,{\sigma_{n,1}}^{3}\sigma_{n,3}\\[1mm]
&\hspace{4mm}+{ \frac
{567}{400}}\,{\sigma_{n,1}}^{2}{\sigma_{n,3}}^{2}+{\frac {189}{
200}}\,{\sigma_{n,2}}^{3}\sigma_{n,3}+{\frac
{567}{400}}\,{\sigma_{n,2 }}^{2}{\sigma_{n,3}}^{2}-{\frac
{27}{8}}\,\sigma_{n,1}{\sigma_{n,2}}^{ 3}+{\frac
{567}{200}}\,\sigma_{n,1}\sigma_{n,2}{\sigma_{n,3}}^{2}-{ \frac
{54663}{25000}}+{\frac {63}{200}}\,{\sigma_{n,3}}^{2}\\[1mm]
&\hspace{4mm}+{\frac {
189}{200}}\,\sigma_{n,1}{\sigma_{n,3}}^{3}+\frac94\,{\sigma_{n,1}}^{2}+{
\frac {15}{4}}\,{\sigma_{n,1}}^{3}+{\frac
{27}{16}}\,{\sigma_{n,1}}^{4 }+{\frac
{189}{800}}\,{\sigma_{n,3}}^{4}+{\frac {567}{200}}\,\sigma_{{
1}}{\sigma_{n,2}}^{2}\sigma_{n,3}+{\frac
{189}{200}}\,\sigma_{n,2}{ \sigma_{n,3}}^{3}+{\frac
{21}{40}}\,{\sigma_{n,3}}^{3}\\[1mm]
&\hspace{4mm}-{\frac {9}{8}} \,{\sigma_{n,2}}^{2}-{\frac
{15}{8}}\,{\sigma_{n,2}}^{3}-{\frac {27}{ 32}}\,{\sigma_{n,2}}^{4}
,\\[1mm]
b_{33}(\sigma_n)&=-{\frac
{63}{25}}\,\sigma_{n,1}\sigma_{n,2}\sigma_{n,3}-{\frac {567}{
200}}\,{\sigma_{n,1}}^{2}\sigma_{n,2}\sigma_{n,3}+{\frac
{45663}{50000 }}+\frac92\,\sigma_{n,1}{\sigma_{n,2}}^{2}-{\frac
{63}{200}}\,\sigma_{n,2} \sigma_{n,3}+{\frac
{9}{8}}\,\sigma_{n,1}\sigma_{n,2}+\frac92\,{\sigma_{{n,1
}}}^{2}\sigma_{n,2}-{\frac {63}{200}}\,\sigma_{n,1}\sigma_{n,3}\\[1mm]
&\hspace{4mm}-{ \frac
{63}{50}}\,{\sigma_{n,1}}^{2}\sigma_{n,3}-{\frac {63}{50}}\,
\sigma_{n,2}{\sigma_{n,3}}^{2}-{\frac
{63}{50}}\,{\sigma_{n,2}}^{2} \sigma_{n,3}-{\frac
{63}{50}}\,\sigma_{n,1}{\sigma_{n,3}}^{2}+{\frac {
27}{8}}\,{\sigma_{n,1}}^{3}\sigma_{n,2}+{\frac
{81}{16}}\,{\sigma_{n,1 }}^{2}{\sigma_{n,2}}^{2}-{\frac
{189}{200}}\,{\sigma_{n,1}}^{3}\sigma_ {{3}}\\[1mm]
&\hspace{4mm}-{\frac
{567}{400}}\,{\sigma_{n,1}}^{2}{\sigma_{n,3}}^{2}-{\frac
{189}{200}}\,{\sigma_{n,2}}^{3}\sigma_{n,3}-{\frac {567}{400}}\,{
\sigma_{n,2}}^{2}{\sigma_{n,3}}^{2}+{\frac {27}{8}}\,\sigma_{n,1}{
\sigma_{n,2}}^{3}-{\frac
{567}{200}}\,\sigma_{n,1}\sigma_{n,2}{\sigma_ {{3}}}^{2}-{\frac
{63}{400}}\,{\sigma_{n,3}}^{2}-{\frac {189}{200}}\,
\sigma_{n,1}{\sigma_{n,3}}^{3}\\[1mm]
&\hspace{4mm}-{\frac {9}{8}}\,{\sigma_{n,1}}^{2}-3\,{
\sigma_{n,1}}^{3}-{\frac {27}{16}}\,{\sigma_{n,1}}^{4}-{\frac
{189}{ 800}}\,{\sigma_{n,3}}^{4}-{\frac
{567}{200}}\,\sigma_{n,1}{\sigma_{n,2 }}^{2}\sigma_{n,3}-{\frac
{189}{200}}\,\sigma_{n,2}{\sigma_{n,3}}^{3}- {\frac
{21}{50}}\,{\sigma_{n,3}}^{3}+{\frac {9}{16}}\,{\sigma_{n,2}}^{
2}+\frac32\,{\sigma_{n,2}}^{3}+{\frac
{27}{32}}\,{\sigma_{n,2}}^{4}
,\\[2mm]
b_{34}(\sigma_n)&={\frac
{63}{100}}\,\sigma_{n,1}\sigma_{n,2}\sigma_{n,3}-{\frac {4623}{
25000}}+{\frac
{189}{200}}\,{\sigma_{n,1}}^{2}\sigma_{n,2}\sigma_{n,3} -{\frac
{9}{8}}\,\sigma_{n,1}{\sigma_{n,2}}^{2}+{\frac {7}{100}}\,
\sigma_{n,2}\sigma_{n,3}-\frac14\,\sigma_{n,1}\sigma_{n,2}-{\frac
{9}{8}} \,{\sigma_{n,1}}^{2}\sigma_{n,2}+{\frac
{7}{100}}\,\sigma_{n,1}\sigma_ {{3}}\\[1mm]
&\hspace{4mm}+{\frac
{63}{200}}\,{\sigma_{n,1}}^{2}\sigma_{n,3}+{\frac {63}{
200}}\,\sigma_{n,2}{\sigma_{n,3}}^{2}+{\frac
{63}{200}}\,{\sigma_{n,2} }^{2}\sigma_{n,3}+{\frac
{63}{200}}\,\sigma_{n,1}{\sigma_{n,3}}^{2}-{ \frac
{9}{8}}\,{\sigma_{n,1}}^{3}\sigma_{n,2}-{\frac {27}{16}}\,{
\sigma_{n,1}}^{2}{\sigma_{n,2}}^{2}+{\frac
{63}{200}}\,{\sigma_{n,1}}^ {3}\sigma_{n,3}\\[1mm]
&\hspace{4mm}+{\frac
{189}{400}}\,{\sigma_{n,1}}^{2}{\sigma_{n,3}}^{ 2}+{\frac
{63}{200}}\,{\sigma_{n,2}}^{3}\sigma_{n,3}+{\frac {189}{400}
}\,{\sigma_{n,2}}^{2}{\sigma_{n,3}}^{2}-{\frac
{9}{8}}\,\sigma_{n,1}{ \sigma_{n,2}}^{3}+{\frac
{189}{200}}\,\sigma_{n,1}\sigma_{n,2}{\sigma_ {{3}}}^{2}+{\frac
{7}{200}}\,{\sigma_{n,3}}^{2}+{\frac {63}{200}}\,
\sigma_{n,1}{\sigma_{n,3}}^{3}\\[1mm]
&\hspace{4mm}+\frac14\,{\sigma_{n,1}}^{2}+\frac34\,{\sigma_{{
1}}}^{3}+{\frac {9}{16}}\,{\sigma_{n,1}}^{4}+{\frac {63}{800}}\,{
\sigma_{n,3}}^{4}+{\frac
{189}{200}}\,\sigma_{n,1}{\sigma_{n,2}}^{2} \sigma_{n,3}+{\frac
{63}{200}}\,\sigma_{n,2}{\sigma_{n,3}}^{3}+{\frac
{21}{200}}\,{\sigma_{n,3}}^{3}-\frac18\,{\sigma_{n,2}}^{2}-\frac38\,{\sigma_{{
2}}}^{3}-{\frac {9}{32}}\,{\sigma_{n,2}}^{4},\\[1mm]
b_{41}(\sigma_n)&=-{\frac
{63}{50}}\,\sigma_{n,1}\sigma_{n,2}\sigma_{n,3}-{\frac {189}{
200}}\,{\sigma_{n,1}}^{2}\sigma_{n,2}\sigma_{n,3}-{\frac
{27}{4}}\, \sigma_{n,1}{\sigma_{n,2}}^{2}-{\frac
{77}{200}}\,\sigma_{n,2}\sigma_{ {3}}-{\frac
{33}{8}}\,\sigma_{n,1}\sigma_{n,2}-{\frac {27}{4}}\,{
\sigma_{n,1}}^{2}\sigma_{n,2}-{\frac
{77}{200}}\,\sigma_{n,1}\sigma_{{ 3}}\\[1mm]
&\hspace{4mm}-{\frac
{63}{100}}\,{\sigma_{n,1}}^{2}\sigma_{n,3}-{\frac {63}{100}
}\,\sigma_{n,2}{\sigma_{n,3}}^{2}-{\frac
{63}{100}}\,{\sigma_{n,2}}^{2 }\sigma_{n,3}-{\frac
{63}{100}}\,\sigma_{n,1}{\sigma_{n,3}}^{2}-{ \frac
{27}{8}}\,{\sigma_{n,1}}^{3}\sigma_{n,2}-{\frac {81}{16}}\,{
\sigma_{n,1}}^{2}{\sigma_{n,2}}^{2}-{\frac
{63}{200}}\,{\sigma_{n,1}}^ {3}\sigma_{n,3}\\[1mm]
&\hspace{4mm}-{\frac
{189}{400}}\,{\sigma_{n,1}}^{2}{\sigma_{n,3}}^{ 2}-{\frac
{63}{200}}\,{\sigma_{n,2}}^{3}\sigma_{n,3}-{\frac {189}{400}
}\,{\sigma_{n,2}}^{2}{\sigma_{n,3}}^{2}-{\frac
{27}{8}}\,\sigma_{n,1}{ \sigma_{n,2}}^{3}-{\frac
{189}{200}}\,\sigma_{n,1}\sigma_{n,2}{\sigma_ {{3}}}^{2}-{\frac
{77}{400}}\,{\sigma_{n,3}}^{2}-{\frac {63}{200}}\,
\sigma_{n,1}{\sigma_{n,3}}^{3}\\[1mm]
&\hspace{4mm}+{\frac {72909}{50000}}-\frac12\,\sigma_{n,1
}-\frac34\,\sigma_{n,2}-{\frac {7}{100}}\,\sigma_{n,3}-{\frac
{11}{8}}\,{ \sigma_{n,1}}^{2}-\frac32\,{\sigma_{n,1}}^{3}-{\frac
{9}{16}}\,{\sigma_{{n,1 }}}^{4}-{\frac
{63}{800}}\,{\sigma_{n,3}}^{4}-{\frac {189}{200}}\,
\sigma_{n,1}{\sigma_{n,2}}^{2}\sigma_{n,3}\\[1mm]
&\hspace{4mm}-{\frac {63}{200}}\,\sigma_{
{2}}{\sigma_{n,3}}^{3}-{\frac
{21}{100}}\,{\sigma_{n,3}}^{3}-{\frac {
33}{16}}\,{\sigma_{n,2}}^{2}-\frac94\,{\sigma_{n,2}}^{3}-{\frac
{27}{32}} \,{\sigma_{n,2}}^{4} ,\\[1mm]
b_{42}(\sigma_n)&={\frac
{63}{20}}\,\sigma_{n,1}\sigma_{n,2}\sigma_{n,3}+{\frac {567}{
200}}\,{\sigma_{n,1}}^{2}\sigma_{n,2}\sigma_{n,3}+{\frac
{135}{8}}\, \sigma_{n,1}{\sigma_{n,2}}^{2}+{\frac
{63}{100}}\,\sigma_{n,2}\sigma_{ {3}}+{\frac
{27}{4}}\,\sigma_{n,1}\sigma_{n,2}+{\frac {135}{8}}\,{
\sigma_{n,1}}^{2}\sigma_{n,2}+{\frac
{63}{100}}\,\sigma_{n,1}\sigma_{{ 3}}\\[1mm]
&\hspace{4mm}+{\frac
{63}{40}}\,{\sigma_{n,1}}^{2}\sigma_{n,3}+{\frac {63}{40}}
\,\sigma_{n,2}{\sigma_{n,3}}^{2}+{\frac
{63}{40}}\,{\sigma_{n,2}}^{2} \sigma_{n,3}+{\frac
{63}{40}}\,\sigma_{n,1}{\sigma_{n,3}}^{2}+{\frac {
81}{8}}\,{\sigma_{n,1}}^{3}\sigma_{n,2}+{\frac
{243}{16}}\,{\sigma_{{n,1 }}}^{2}{\sigma_{n,2}}^{2}+{\frac
{189}{200}}\,{\sigma_{n,1}}^{3}\sigma _{{3}}\\[1mm]
&\hspace{4mm}+{\frac
{567}{400}}\,{\sigma_{n,1}}^{2}{\sigma_{n,3}}^{2}+{ \frac
{189}{200}}\,{\sigma_{n,2}}^{3}\sigma_{n,3}+{\frac {567}{400}}\,
{\sigma_{n,2}}^{2}{\sigma_{n,3}}^{2}+{\frac
{81}{8}}\,\sigma_{n,1}{ \sigma_{n,2}}^{3}+{\frac
{567}{200}}\,\sigma_{n,1}\sigma_{n,2}{\sigma_ {{3}}}^{2}-{\frac
{54663}{25000}}+{\frac {63}{200}}\,{\sigma_{n,3}}^{2 }\\[1mm]
&\hspace{4mm}+{\frac
{189}{200}}\,\sigma_{n,1}{\sigma_{n,3}}^{3}+\frac94\,{\sigma_{n,1
}}^{2}+{\frac {15}{4}}\,{\sigma_{n,1}}^{3}+{\frac
{27}{16}}\,{\sigma_{ {1}}}^{4}+{\frac
{189}{800}}\,{\sigma_{n,3}}^{4}+{\frac {567}{200}}\,
\sigma_{n,1}{\sigma_{n,2}}^{2}\sigma_{n,3}+{\frac
{189}{200}}\,\sigma_ {{2}}{\sigma_{n,3}}^{3}+{\frac
{21}{40}}\,{\sigma_{n,3}}^{3}+{\frac {
27}{8}}\,{\sigma_{n,2}}^{2}\\[1mm]
&\hspace{4mm}+{\frac {45}{8}}\,{\sigma_{n,2}}^{3}+{ \frac
{81}{32}}\,{\sigma_{n,2}}^{4} ,
\end{align*}}
{\footnotesize\begin{align*} b_{43}(\sigma_n)&=-{\frac
{63}{25}}\,\sigma_{n,1}\sigma_{n,2}\sigma_{n,3}-{\frac {567}{
200}}\,{\sigma_{n,1}}^{2}\sigma_{n,2}\sigma_{n,3}+{\frac
{45663}{50000 }}-{\frac
{27}{2}}\,\sigma_{n,1}{\sigma_{n,2}}^{2}-{\frac {63}{200}}\,
\sigma_{n,2}\sigma_{n,3}-{\frac
{27}{8}}\,\sigma_{n,1}\sigma_{n,2}-{ \frac
{27}{2}}\,{\sigma_{n,1}}^{2}\sigma_{n,2}\\[1mm]
&\hspace{4mm}-{\frac {63}{200}}\, \sigma_{n,1}\sigma_{n,3}-{\frac
{63}{50}}\,{\sigma_{n,1}}^{2}\sigma_{{ 3}}-{\frac
{63}{50}}\,\sigma_{n,2}{\sigma_{n,3}}^{2}-{\frac {63}{50}}
\,{\sigma_{n,2}}^{2}\sigma_{n,3}-{\frac
{63}{50}}\,\sigma_{n,1}{\sigma _{{3}}}^{2}-{\frac
{81}{8}}\,{\sigma_{n,1}}^{3}\sigma_{n,2}-{\frac {
243}{16}}\,{\sigma_{n,1}}^{2}{\sigma_{n,2}}^{2}\\[1mm]
&\hspace{4mm}-{\frac {189}{200}}\,{
\sigma_{n,1}}^{3}\sigma_{n,3}-{\frac
{567}{400}}\,{\sigma_{n,1}}^{2}{ \sigma_{n,3}}^{2}-{\frac
{189}{200}}\,{\sigma_{n,2}}^{3}\sigma_{n,3}-{ \frac
{567}{400}}\,{\sigma_{n,2}}^{2}{\sigma_{n,3}}^{2}-{\frac {81}{8}
}\,\sigma_{n,1}{\sigma_{n,2}}^{3}-{\frac {567}{200}}\,\sigma_{n,1}
\sigma_{n,2}{\sigma_{n,3}}^{2}\\[1mm]
&\hspace{4mm}-{\frac {63}{400}}\,{\sigma_{n,3}}^{2}-{ \frac
{189}{200}}\,\sigma_{n,1}{\sigma_{n,3}}^{3}-{\frac {9}{8}}\,{
\sigma_{n,1}}^{2}-3\,{\sigma_{n,1}}^{3}-{\frac
{27}{16}}\,{\sigma_{n,1 }}^{4}-{\frac
{189}{800}}\,{\sigma_{n,3}}^{4}-{\frac {567}{200}}\,
\sigma_{n,1}{\sigma_{n,2}}^{2}\sigma_{n,3}-{\frac
{189}{200}}\,\sigma_ {{2}}{\sigma_{n,3}}^{3}\\[2mm]
&\hspace{4mm}-{\frac {21}{50}}\,{\sigma_{n,3}}^{3}-{\frac {
27}{16}}\,{\sigma_{n,2}}^{2}-\frac92\,{\sigma_{n,2}}^{3}-{\frac
{81}{32}} \,{\sigma_{n,2}}^{4}
,\\[2mm]
b_{44}(\sigma_n)&={\frac
{27}{8}}\,\sigma_{{n,1}}{\sigma_{{n,2}}}^{2}-{\frac
{4623}{25000}}+ {\frac {7}{200}}\,{\sigma_{{n,3}}}^{2}+{\frac
{63}{800}}\,{\sigma_{{n,3}}} ^{4}+{\frac
{21}{200}}\,{\sigma_{{n,3}}}^{3}+{\frac {63}{200}}\,\sigma_{
{n,1}}{\sigma_{{n,3}}}^{3}+{\frac
{7}{100}}\,\sigma_{{n,1}}\sigma_{{n,3}}+{ \frac
{63}{200}}\,\sigma_{{n,2}}{\sigma_{{n,3}}}^{3}\\[2mm]
&\hspace{4mm}+{\frac {27}{32}}\,{ \sigma_{{n,2}}}^{4}+{\frac
{9}{8}}\,{\sigma_{{n,2}}}^{3}+\frac38\,{\sigma_{{n,2}
}}^{2}+\frac34\,\sigma_{{n,1}}\sigma_{{n,2}}+{\frac
{63}{200}}\,{\sigma_{{n,2}}} ^{2}\sigma_{{n,3}}+{\frac
{63}{200}}\,\sigma_{{n,1}}{\sigma_{{n,3}}}^{2}+{ \frac
{189}{400}}\,{\sigma_{{n,2}}}^{2}{\sigma_{{n,3}}}^{2}+{\frac {63}{
200}}\,{\sigma_{{n,2}}}^{3}\sigma_{{n,3}}
\\[2mm]
&\hspace{4mm}+{\frac {189}{400}}\,{\sigma_{{n,1}
}}^{2}{\sigma_{{n,3}}}^{2}+{\frac
{63}{200}}\,{\sigma_{{n,1}}}^{3}\sigma_{ {3}}+{\frac
{63}{200}}\,{\sigma_{{n,1}}}^{2}\sigma_{{n,3}}+{\frac {7}{100}
}\,\sigma_{{n,2}}\sigma_{{n,3}}
+{\frac
{27}{8}}\,{\sigma_{{n,1}}}^{2}\sigma_ {{n,2}}+{\frac
{27}{8}}\,\sigma_{{n,1}}{\sigma_{{n,2}}}^{3}+{\frac {81}{16}}
\,{\sigma_{{n,1}}}^{2}{\sigma_{{n,2}}}^{2}\\[2mm]
&\hspace{4mm}+{\frac {27}{8}}\,{\sigma_{{n,1}}}
^{3}\sigma_{{n,2}}+{\frac
{63}{200}}\,\sigma_{{n,2}}{\sigma_{{n,3}}}^{2}+{ \frac
{189}{200}}\,{\sigma_{{n,1}}}^{2}\sigma_{{n,2}}\sigma_{{n,3}}+{\frac
{
189}{200}}\,\sigma_{{n,1}}{\sigma_{{n,2}}}^{2}\sigma_{{n,3}}+{\frac
{189}{
200}}\,\sigma_{{n,1}}\sigma_{{n,2}}{\sigma_{{n,3}}}^{2}+{\frac
{63}{100}}\,
\sigma_{{n,1}}\sigma_{{n,2}}\sigma_{{n,3}}\\[2mm]
&\hspace{4mm}+\frac14\,{\sigma_{{n,1}}}^{2}+\frac34\,{
\sigma_{{n,1}}}^{3}+{\frac {9}{16}}\,{\sigma_{{n,1}}}^{4} .
\end{align*}}
\setcounter{equation}{0} \setcounter{definition}{0}
\setcounter{theorem}{0}
\section{Numerical Experiments}\label{Sec4}
In this section, we present the results of numerical experiments
for the VS SDIMSIMs of orders $p\leq4$ constructed in the previous
section. The utilized stepsize pattern is according to the
following formula
\[
h_{n+1} = \rho^{(-1)^n \sin(5\pi n/(X-x_0))}\cdot h_n,
\]
with $h_0=(X-x_0)/N$, which changes the stepsize rapidly. We
should say that the resulting grid $x_n$, $n=0,1,\cdots,N$, is
uniformly rescaled so that $x_N=X$. In our numerical results, we
report $\texttt{ge}$ as the error of the methods at the point $X$
measured in the maximum norm. Moreover, we provide a numerical
estimation to the order of convergence, $p$ which is computed by
the formula
\[
O_N:=\log\Bigl(\frac{\texttt{ge}_1}{\texttt{ge}_2}\Bigr)\Big/\log
\Bigl(\frac{N_2}{N_1}\Bigr),
\]
where $\texttt{ge}_1$ and $\texttt{ge}_2$ respectively stand for
the global error of the methods corresponding to $N_1$ and $N_2$
grid points.

As a first example, we consider the following linear nonstiff ODEs
\begin{align}\label{eq:lin}
\left[
\begin{array}{c}
y'_1(x) \\[1mm]
y'_2(x)
\end{array}
\right]=\left[
\begin{array}{cc}
\phantom{-}1 & \phantom{-}1 \\
-2 & -1
\end{array}
\right]\left[
\begin{array}{c}
y_1(x) \\[1mm]
y_2(x)
\end{array}
\right], \quad \left[
\begin{array}{c}
y_1(0) \\[1mm]
y_2(0)
\end{array}
\right]=\left[
\begin{array}{c}
2 \\
1
\end{array}
\right],
\end{align}
with $X=5\pi$ whose the exact solution is given by
\[
y_1(x)=3\sin(x)+2\cos(x), \quad y_2(x)=\cos(x)-5\sin(x).
\]
The numerical results with $\rho=2$ and $\rho=4$ for this problem
are respectively reported in Tables \ref{Tab:eq:lin2} and
\ref{Tab:eq:lin4}. These results show the high accuracy of the
methods and good agreements with their theoretical orders.

\begin{table}[h]
\begin{center}
\caption{The numerical results of the methods with $\rho=2$ for
the problem~\eqref{eq:lin}.}\label{Tab:eq:lin2}{\footnotesize
\begin{tabular}{lcccccc}
\hline
& & & & & &\\[-2.5mm]
$N$ & & $1000$ & $2000$ & $4000$ & $8000$ & $16000$\\[0.3mm]
\hline
& & & & & &\\[-2.5mm]
\multirow{2}{*}{Method with $p=1$} & $\texttt{ge}$ &
$4.71\times10^{-3}$ & $1.21\times10^{-3}$ & $3.34\times10^{-4}$ &
$1.07\times10^{-4}$ &
$4.25\times10^{-5}$\\[0.3mm]
& $O_N$ & & $1.96$ &
$1.86$ & $1.64$ & $1.33$\\[3mm]
\multirow{2}{*}{Method with $p=2$} & $\texttt{ge}$ &
$3.53\times10^{-4}$ & $8.83\times10^{-5}$ & $2.21\times10^{-5}$ &
$5.51\times10^{-6}$ &
$1.38\times10^{-6}$\\[0.3mm]
& $O_N$ & & $2.00$ &
$2.00$ & $2.00$ & $2.00$\\[3mm]
\multirow{2}{*}{Method with $p=3$} & $\texttt{ge}$ &
$1.06\times10^{-5}$ & $1.30\times10^{-6}$ & $1.63\times10^{-7}$ &
$2.04\times10^{-8}$ &
$2.55\times10^{-9}$\\[0.3mm]
& $O_N$ & & $2.20$ &
$3.82$ & $3.00$ & $3.00$\\[3mm]
\multirow{2}{*}{Method with $p=4$} & $\texttt{ge}$ &
$1.64\times10^{-8}$ & $9.74\times10^{-10}$ & $6.21\times10^{-11}$
& $4.24\times10^{-12}$ &
$1.36\times10^{-12}$\\[0.3mm]
& $O_N$ & & $4.07$ &
$3.97$ & $3.87$ & $1.64$\\[0.3mm]
\hline
\end{tabular}}
\end{center}
\end{table}

\begin{table}[h]
\begin{center}
\caption{The numerical results of the methods with $\rho=4$ for
the problem~\eqref{eq:lin}.}\label{Tab:eq:lin4}{\footnotesize
\begin{tabular}{lcccccc}
\hline
& & & & & &\\[-2.5mm]
$N$ & & $1000$ & $2000$ & $4000$ & $8000$ & $16000$\\[0.3mm]
\hline
& & & & & &\\[-2.5mm]
\multirow{2}{*}{Method with $p=1$} & $\texttt{ge}$ &
$7.22\times10^{-3}$ & $1.84\times10^{-3}$ & $4.91\times10^{-4}$ &
$1.47\times10^{-4}$ &
$5.39\times10^{-5}$\\[0.3mm]
& $O_N$ & & $1.97$ &
$1.91$ & $1.73$ & $1.45$\\[3mm]
\multirow{2}{*}{Method with $p=2$} & $\texttt{ge}$ &
$1.46\times10^{-}3$ & $3.65\times10^{-4}$ & $9.13\times10^{-5}$ &
$2.28\times10^{-5}$ &
$5.71\times10^{-6}$\\[0.3mm]
& $O_N$ & & $2.00$ &
$2.00$ & $2.00$ & $2.00$\\[3mm]
\multirow{2}{*}{Method with $p=3$} & $\texttt{ge}$ &
$4.59\times10^{-5}$ & $5.42\times10^{-6}$ & $6.71\times10^{-7}$ &
$8.38\times10^{-8}$ &
$1.05\times10^{-8}$\\[0.3mm]
& $O_N$ & & $3.08$ &
$3.01$ & $3.00$ & $3.00$\\[3mm]
\multirow{2}{*}{Method with $p=4$} & $\texttt{ge}$ &
$1.04\times10^{-7}$ & $6.40\times10^{-9}$ & $4.01\times10^{-10}$ &
$2.60\times10^{-11}$ &
$2.62\times10^{-12}$\\[0.3mm]
& $O_N$ & & $4.02$ &
$4.00$ & $3.95$ & $3.31$\\[0.3mm]
\hline
\end{tabular}}
\end{center}
\end{table}

As the next example, we consider the well-known Brusselator
\cite{hnw93,Jackbook2009}
\begin{align}\label{eq:bruss}
\left\{%
\begin{array}{ll}
  y_1'=1+y_1^2y_2-4y_1, & y_1(0)=1.5, \\[2mm]
  y_2'=3y_1-y_1^2y_2,& y_2(0)=3,\\
\end{array}%
\right.
\end{align}
with integration interval $[0,20]$. To compute the global error of
the methods, we use the reference solution obtained by the code
\texttt{ode45} from the \textsc{Matlab} ODE suite \cite{sr97} with
tolerances $Atol=10^{-14}$ and $Rtol=2.22045\times10^{-14}$. In
Tables \ref{Tab:eq:bruss2} and \ref{Tab:eq:bruss4}, we repeat the
numerical results reported in Tables \ref{Tab:eq:lin2} and
\ref{Tab:eq:lin4} but now for the example \eqref{eq:bruss}. Again,
the results verify the high accuracy of the methods and the errors
decrease with orders which coincide with the theoretical
expectations.

\begin{table}[h]
\begin{center}
\caption{The numerical results of the methods with $\rho=2$ for
the problem~\eqref{eq:bruss}.}\label{Tab:eq:bruss2}{\footnotesize
\begin{tabular}{lcccccc}
\hline
& & & & & &\\[-2.5mm]
$N$ & & $1000$ & $2000$ & $4000$ & $8000$ & $16000$\\[0.3mm]
\hline
& & & & & &\\[-2.5mm]
\multirow{2}{*}{Method with $p=1$} & $\texttt{ge}$ &
$3.43\times10^{-4}$ & $9.90\times10^{-5}$ & $3.00\times10^{-5}$ &
$9.94\times10^{-6}$ &
$3.68\times10^{-6}$\\[0.3mm]
& $O_N$ & & $1.79$ &
$1.72$ & $1.59$ & $1.43$\\[3mm]
\multirow{2}{*}{Method with $p=2$} & $\texttt{ge}$ &
$1.41\times10^{-5}$ & $5.23\times10^{-6}$ & $1.48\times10^{-6}$ &
$3.90\times10^{-7}$ &
$9.98\times10^{-8}$\\[0.3mm]
& $O_N$ & & $1.43$ &
$1.82$ & $1.92$ & $1.97$\\[3mm]
\multirow{2}{*}{Method with $p=3$} & $\texttt{ge}$ &
$1.92\times10^{-5}$ & $2.01\times10^{-6}$ & $2.21\times10^{-7}$ &
$2.55\times10^{-8}$ &
$3.06\times10^{-9}$\\[0.3mm]
& $O_N$ & & $3.26$ &
$3.19$ & $3.12$ & $3.06$\\[3mm]
\multirow{2}{*}{Method with $p=4$} & $\texttt{ge}$ &
$3.29\times10^{-6}$ & $3.19\times10^{-8}$ & $6.43\times10^{-10}$ &
$2.01\times10^{-11}$ &
$1.04\times10^{-12}$\\[0.3mm]
& $O_N$ & & $6.69$ &
$5.63$ & $5.00$ & $4.27$\\[0.3mm]
\hline
\end{tabular}}
\end{center}
\end{table}

\begin{table}[h]
\begin{center}
\caption{The numerical results of the methods with $\rho=4$ for
the problem~\eqref{eq:bruss}.}\label{Tab:eq:bruss4}{\footnotesize
\begin{tabular}{lcccccc}
\hline
& & & & & &\\[-2.5mm]
$N$ & & $1000$ & $2000$ & $4000$ & $8000$ & $16000$\\[0.3mm]
\hline
& & & & & &\\[-2.5mm]
\multirow{2}{*}{Method with $p=1$} & $\texttt{ge}$ &
$4.88\times10^{-4}$ & $1.40\times10^{-4}$ & $4.16\times10^{-5}$ &
$1.33\times10^{-5}$ &
$4.74\times10^{-6}$\\[0.3mm]
& $O_N$ & & $1.80$ &
$1.75$ & $1.65$ & $1.49$\\[3mm]
\multirow{2}{*}{Method with $p=2$} & $\texttt{ge}$ &
$7.18\times10^{-5}$ & $2.29\times10^{-5}$ & $6.24\times10^{-6}$ &
$1.62\times10^{-6}$ &
$4.11\times10^{-7}$\\[0.3mm]
& $O_N$ & & $1.65$ &
$1.88$ & $1.95$ & $1.98$\\[3mm]
\multirow{2}{*}{Method with $p=3$} & $\texttt{ge}$ &
$1.04\times10^{-4}$ & $9.87\times10^{-6}$ & $1.02\times10^{-6}$ &
$1.13\times10^{-7}$ &
$1.31\times10^{-8}$\\[0.3mm]
& $O_N$ & & $3.40$ &
$3.27$ & $3.17$ & $3.11$\\[3mm]
\multirow{2}{*}{Method with $p=4$} & $\texttt{ge}$ &
$7.72\times10^{-4}$ & $8.89\times10^{-8}$ & $1.40\times10^{-9}$ &
$1.82\times10^{-11}$ &
$9.71\times10^{-13}$\\[0.3mm]
& $O_N$ & & $13.08$ &
$5.99$ & $6.27$ & $4.23$\\[0.3mm]
\hline
\end{tabular}}
\end{center}
\end{table}

As the final example, we study the reaction-diffusion equation
(Brusselator with diffusion) \citep{Hairer},
\begin{equation}
\begin{aligned}\label{eq:SBRUSS}
\frac{\partial u}{\partial t}&=A+u^2v-(B+1)u+\alpha \frac{\partial ^2u}{\partial x^2},\\[2mm]
\frac{\partial v}{\partial t}&=Bu-u^2v+\alpha \frac{\partial
^2v}{\partial x^2},
\end{aligned}
\end{equation}
with $0\leq x \leq 1$, $A=1$, $B=3$, $\alpha=1/50$ and boundary
conditions
\begin{align*}
u(0,t)&=u(1,t)=1,\qquad v(0,t)=v(1,t)=3,\\
u(x,0)&=1+\sin(2\pi x), \qquad  \hspace*{-3mm}v(x,0)=3.
\end{align*}
The system of ODEs resulting from discretization of the diffusion
terms by the method of lines on a grid of $N=50$ points
$x_i=i/(N+1)$, $i=1,2,\ldots,N$, $\Delta x=1/(N+1)$, which is a
system of dimension $100$, takes the form
\begin{align*}
u'_i&=1+u_i^2v_i-4u_i+\frac{\alpha }{(\Delta x)^2}(u_{i-1}-2u_i+u_{i+1}),\\[2mm]
v'_i&=3u_i-u_i^2v_i+\frac{\alpha }{(\Delta x)^2}(v_{i-1}-2v_i+v_{i+1}),\\[2mm]
u_0(t)&=u_{N+1}(t)=1,\qquad v_0(t)=v_{N+1}(t)=3,\\[2mm]
u_i(0)&=1+\sin(2\pi x_i), \qquad  \hspace*{-3mm}v_i(0)=3,\quad
i=1,2,\ldots,N.
\end{align*}
The numerical results for this problem  with $\rho=2$ at the end
point of integration $t_{end}=10$ are given in Table
\ref{Tab:eq:Sbruss2}. n order to compute the global error of the
methods, again we use the reference solution obtained by the code
\texttt{ode45} from the \textsc{Matlab}. These results show the
capability and efficiency of the proposed methods in solving
problems in higher dimensions with expected orders.

\begin{table}[h]
\begin{center}
\caption{The numerical results of the methods with $\rho=2$ for
the
problem~\eqref{eq:SBRUSS}.}\label{Tab:eq:Sbruss2}{\footnotesize
\begin{tabular}{lcccccc}
\hline
& & & & & &\\[-2.5mm]
$N$ & & $12000$ & $13000$ & $14000$ & $15000$ & $16000$\\[0.3mm]
\hline
& & & & & &\\[-2.5mm]
\multirow{2}{*}{Method with $p=1$} & $\texttt{ge}$ &
$6.27\times10^{-6}$ & $5.51\times10^{-6}$ & $4.91\times10^{-6}$ &
$4.42\times10^{-6}$ &
$4.01\times10^{-6}$\\[0.3mm]
& $O_N$ & & $1.61$ &
$1.56$ & $1.52$ & $1.51$\\[3mm]
\multirow{2}{*}{Method with $p=2$} & $\texttt{ge}$ &
$3.03\times10^{-7}$ & $2.59\times10^{-7}$ & $2.23\times10^{-7}$ &
$1.94\times10^{-7}$ &
$1.71\times10^{-7}$\\[0.3mm]
& $O_N$ & & $1.96$ &
$2.02$ & $2.02$ & $1.96$\\[3mm]
\multirow{2}{*}{Method with $p=3$} & $\texttt{ge}$ &
$3.24\times10^{-9}$ & $2.54\times10^{-9}$ & $2.02\times10^{-9}$ &
$1.64\times10^{-9}$ &
$1.35\times10^{-9}$\\[0.3mm]
& $O_N$ & & $3.04$ &
$3.09$ & $3.02$ & $3.01$\\[3mm]
\multirow{2}{*}{Method with $p=4$} & $\texttt{ge}$ &
$1.15\times10^{-11}$ & $8.04\times10^{-12}$ & $5.99\times10^{-12}$
& $4.81\times10^{-12}$ &
$3.72\times10^{-12}$\\[0.3mm]
& $O_N$ & & $4.47$ &
$3.97$ & $3.18$ & $3.98$\\[0.3mm]
\hline
\end{tabular}}
\end{center}
\end{table}

As it has been described in  \cite{AbdCon}, the implementation of
SGLMs using Nordsieck technique, in addition to computing the
output vector, requires to approximate the Nordsieck vector at the
end of each step; then the input vector for the next step is
computed by rescaling the approximated Nordsieck vector and its
left production by an appropriate matrix; while this is not the
case for the proposed VS SDIMSIMs in which the input vector for
the next step exactly is the same output vector of the current
step. This reduces the number of required floating point
operations.
\setcounter{equation}{0} \setcounter{definition}{0}
\setcounter{theorem}{0}
\section{Conclusions}\label{Sec5}
We introduced the SGLMs in a variable stepsize environment in
which the coefficients matrices of the methods depend on the
rations of the current stepsize and the past stepsizes. By
formulating such methods, we derived their order conditions of
order $p$ and high stage order $q=p$. Construction of such methods
in a special class with unconditionally zero--stability property
for any step size pattern, up to order four, was described.
Finally, some numerical experiments were provided demonstrating
the efficiency and high accuracy of the proposed methods.


\end{document}